\documentclass[12pt,reqno]{amsart}
\usepackage{amsmath, amsthm, amssymb}

\usepackage{comment}
\usepackage{color}

\advance \topmargin by -\headheight
\advance \topmargin by -\headsep
     
\evensidemargin \oddsidemargin
\newtheorem{theorem}{Theorem}[section]
\newtheorem{lemma}[theorem]{Lemma}

\theoremstyle{definition}

\theoremstyle{remark}

\numberwithin{equation}{section}

\newcommand{\mmod}[1]{\,\,\text{mod}\,\,#1}

\def\bfa{{\mathbf a}}
\def\bfb{{\mathbf b}}

 \def\bfe{{\mathbf e}}

\def\bfh{{\mathbf h}}

\def\bfx{{\mathbf x}}
\def\bfy{{\mathbf y}}
\def\bfz{{\mathbf z}}

\def\calB{{\mathcal B}} 
\def\calC{{\mathcal C}}

\def\calI{{\mathcal I}}\def\calItil{{\widetilde \calI}}
\def\calJ{{\mathcal J}}

\def\calS{{\mathcal S}}

\def\A{{\mathbb A}}
\def\N{{\mathbb N}}\def\P{{\mathbb P}}
\def\R{{\mathbb R}}
\def\Z{{\mathbb Z}}\def\Q{{\mathbb Q}}

\def\grC{{\mathfrak C}}

\def\grJ{{\mathfrak J}}

\def\grm{{\mathfrak m}}\def\grM{{\mathfrak M}}

\def\grS{{\mathfrak S}}
\def\grC{{\mathfrak C}}

\def\alp{{\alpha}} 
\def\bet{{\beta}}  
\def\gam{{\gamma}} 

\def\Gam{{\Gamma}}
\def\del{{\delta}}

\def\tet{{\theta}}  

\def\kap{{\kappa}}

\def\bfxi{{\boldsymbol \xi}}

\def\bfkap{{\boldsymbol \kap}}

\def\sig{{\sigma}} \def\Sig{{\Sigma}} 
\def\bftau{{\boldsymbol \tau}}
 
\def\ome{{\omega}} \def\Ome{{\Omega}}
\def\omehat{{\hat{\ome}}}
\def\d{{\partial}}
\def\eps{\varepsilon}

\def\d{{\,{\rm d}}}

\def\bfatil{{\widetilde \bfa}}
\def\bfbtil{{\widetilde \bfb}}
\def\atil{{\tilde a}}
\def\btil{{\tilde b}}

\def\meas{{\rm meas}}

\DeclareMathOperator{\Sing}{Sing}
\DeclareMathOperator{\id}{id}
\DeclareMathOperator{\supp}{supp}

\newenvironment{blue}{\color{blue}}{}

\specialcomment{comblue}{\begin{blue}}{\end{blue}}
\excludecomment{comblue}

\begin{document}
\title[Counting rational points on hypersurfaces]{Counting rational points on hypersurfaces and higher order expansions}
%alternative titles: Higher order expansions via the circle method
\author[Damaris Schindler]{Damaris Schindler}
\address{Hausdorff Center for Mathematics, Endenicher Allee 62-64, 53115 Bonn, Germany}
\email{damaris.schindler@hcm.uni-bonn.de}


\begin{abstract}
We study the number of representations of an integer $n=F(\bfx)$ by a homogeneous form in sufficiently many variables. This is a classical problem in number theory to which the circle method has been succesfully applied to give an asymptotic for the number of such representations where the integer vector $\bfx$ is restricted to a box of side length $P$ for $P$ sufficiently large. In the special case of Waring's problem, Vaughan and Wooley have recently established for the first time a higher order expansion for the corresponding asymptotic formula. Via a different and much more general approach we derive a multi-term asymptotic for this problem for general forms $F(\bfx)$ and give an interpretation for the occurring lower order terms.\par
As an application we derive higher order expansions for the number of rational points of bounded anticanonical height on the projective hypersurface $F(\bfx)=0$ for forms $F(\bfx)$ in sufficiently many variables.
\end{abstract}

\maketitle

%\includecomment{com}
\excludecomment{com}

\section{Introduction}
Let $F(x_1,\ldots, x_s)\in\Z[x_1,\ldots, x_s]$ be a homogeneous form of degree $d$, and let $\calB\subset \R^n$ be a box with sides parallel to the coordinate axes. A natural question in number theory is to study the number of solutions to the equation
\begin{equation*}
F(x_1,\ldots, x_s)=n,
\end{equation*}
with $x_i\in\Z$ for $1\leq i\leq s$ and $(x_1,\ldots, x_s)\in P\calB$ for a large real number $P$. We call this counting function $R_\calB(P,n)$ or when the box $\calB$ is clear from the context also $R(P,n)$, and for convenience we use in the following the vector notation $\bfx$ for $(x_1,\ldots, x_s)$. In the situation where the form $F(\bfx)$ is not too singular and the number of variables $n$ is large compared to the degree $d$, the Hardy-Littlewood circle method has been very successful in providing asymptotics for the counting function $R_\calB(P,n)$. A first very general result in this direction was given by Birch in 1962 \cite{Bir62}. If $\dim V^*$ denotes the affine dimension of the singular locus of the form $F(\bfx)=0$, then Birch proves an asymptotic for $R_\calB(P,n)$ as soon as 
\begin{equation}\label{eqn0.0}
s-\dim V^*>(d-1)2^d.
\end{equation}
There have been many refinements of the method since then, most of them developed for the homogeneous problem with $n=0$ or a weighted version of the counting function $R(P,n)$. These results include for example improvements in the cubic case due to Heath-Brown \cite{HB07}, \cite{HB83}, a new version of the circle method by Heath-Brown in \cite{HB96}, improvements in the quartic case by Browning and Heath-Brown \cite{BroHB09} and improvements on the bound (\ref{eqn0.0}) by Browning and Prendiville \cite{BroPreA14}.\par

All of these have in common that they produce an asymptotic formula for $R(P,n)$ (or a weighted version of this counting function) of the form
\begin{equation}\label{eqn0.1}
R(P,n)= \grS(n)\calJ (P^{-d}n) P^{s-d}+O(P^{s-d-\del}),
\end{equation}
for some positive $\del>0$. The value of $\del$ stays unspecified in most applications of the circle method mentioned above. Here $\grS(n)$ is the singular series and $\calJ(n)$ the singular integral. It is a natural question to ask to what extend one might be able to improve the error term in the asymptotic expansion (\ref{eqn0.1}). An inspection of the arguments in Birch's work \cite{Bir62} shows that the minor arc contribution can be forced to be arbitrarily small when the number of variables $s$ is sufficiently large. However, in the classical major arc analysis there seems to be a natural barrier which prevents one from obtaining any better error term than $P^{s-d-1+\eta}$, for some small $\eta >0$. More rigorously, in \cite{Loh96} Loh has studied the error term in Waring's problem. If $R_s(n)$ denotes the number of representations of some natural number $n$ as the sum of $s$ $k$th powers of positive integers, then provided that $s$ is sufficiently large, the circle method delivers an asymptotic formula of the form
\begin{equation*}
R_s(n)\sim \frac{\Gam(1+1/k)^s}{\Gam(s/k)}\grS_s(n) n^{s/k-1}.
\end{equation*}
Loh shows that for $s\geq k+2$, the error term in this expansion is bounded below by
\begin{equation*}
R_s(n)-\frac{\Gam(1+1/k)^s}{\Gam(s/k)}\grS_s(n) n^{s/k-1}= \Ome_-(n^{(s-1)/k-1}).
\end{equation*}
In their recent work \cite{VauWooA14} Vaughan and Wooley were able to explain this behaviour in establishing second and higher order terms in the asymptotic expansion of $R_s(n)$.\par
Their arguments are very specific for the situation of Waring's problem, which means in our language a diagonal hypersurface of the form 
\begin{equation*}
F(\bfx)=x_1^k+\ldots +x_s^k.
\end{equation*}
It is not clear from their work how to generalise this to more general forms $F(\bfx)$, since they heavily use the additive structure and separation of variables in their proof.
Taking a different approach, we establish in this paper an asymptotic expansion for $R(P,n)$ with an arbitrary number of higher order terms given that the form $F$ is not too singular and the number of variables $s$ is sufficiently large.\par
Suppose we aim to count integer points on the variety $F(\bfx)=n$ in a dilate of a box $\calB= \prod_{i=1}^s (a_i,b_i]$. If we have for example $n=0$ and $F(\bfx)$ is of degree $d$ we expect to see a main term of the form $cP^{s-d}$. If we restricted $F$ to one of the faces of the box, for example in setting $x_i=a_i$ or $x_i=b_i$ for some $1\leq i\leq s$, then we expect that this face contributes a term of order of magnitude $P^{s-d-1}$. In the standard formulations and applications of the circle method this term is not visible since it is typically contained in the error term. We are going to make these terms visible in our main Theorem \ref{thm1} below. The coefficients of these lower order terms, which to some extent correspond to lower dimensional faces of the box $\calB$, are again given by products of generalized singular series and singular integrals. We next describe their shape.\par
Let $K\geq 1$. In order to label the different lower order contributions, we introduce the index set $\calI (K)$, which is defined as follows. Let $\calI (K)$ be the set of all tuples $(I_1,I_2,\bftau)$, where $I_1$ and $I_2$ are disjoint subsets of $\{1,\ldots ,s\}$ and $\bftau\in \Z^s_{\geq 0}$ satisfies $\tau_i=0$ if $i\notin I_1\cup I_2$ and for which the condition $|I_1|+|I_2|+|\bftau|< K$ holds. Furthermore we let $\calI = \cup_{K\geq 1} \calI (K)$ be the union over all these sets. Given a tuple $(I_1,I_2,\bftau)\in \calI$ we will always set $I_3= \{1,\ldots, s\} \setminus (I_1\cup I_2)$ and $I_4=\emptyset$.\par
In a more general situation, if $\{1,\ldots, s\}=\cup_{i=1}^4 I_i$ is a partition into four arbitrary disjoint index sets, and $\bfa,\bfb\in \R^{s}$, then we define the $s$-dimensional vector $\sig_{\bfa,\bfb}(\bfx)$ componentwise by 
\begin{align*}
\sig_{\bfa,\bfb}(\bfx)_i=\left\{\begin{array}{ccc}b_i & \mbox{ if } & i\in I_1 \\ a_i & \mbox{ if }& i\in I_2 \\ x_i & \mbox{ if } & i\in I_3\cup I_4.\end{array}\right.
\end{align*}
For a tuple $(I_1,I_2,\bftau)\in \calI$ we now introduce the integral
\begin{equation*}
J_{(I_1,I_2,\bftau)}(\gam)= \int_{\prod_{i\in I_3}[a_i,b_i]} \partial_\bfx^{\bftau}e(\gam F(\bfx))|_{\bfx= \sig_{\bfa,\bfb}(\bfx)}\d \bfx_{I_3}.
\end{equation*}
The generalized singular integral $\calJ_{(I_1,I_2,\bftau)} (n)$ is then given by
\begin{equation*}
\calJ_{(I_1,I_2,\bftau)} (n)= \int_\R J_{(I_1,I_2,\bftau)}(\gam)e(-\gam n) \d\gam,
\end{equation*}
in case this is convergent.\par
For any $l\geq 0$ let $B_l(x)$ be the $l$th Bernoulli polynomial and set $\bet_l(x)=B_l(\{x\})$. In order to introduce the generalized singular series, we need to introduce some versions of the usual exponential sums occurring in the circle method. For $1\leq r<q$ we define
\begin{equation*}
\begin{split}
S_{(I_1,I_2,\bftau)}(P;r,q)= \sum_{0\leq \bfz <q} &e\left(\frac{r}{q}F(\bfz)\right) \left(\prod_{i\in I_1} \frac{(-1)^{\tau_i+1}}{(\tau_i+1)!} \bet_{\tau_i+1}\left(\frac{Pb_i-z_i}{q}\right)\right) \\ & \left(\prod_{i\in I_2} \frac{(-1)^{\tau_i}}{(\tau_i+1)!} \bet_{\tau_i+1}\left(\frac{Pa_i-z_i}{q}\right)\right).
\end{split}
\end{equation*}
The singular series $\grS_{(I_1,I_2,\bftau)}(P,n)$ is then given by
\begin{equation*}
\grS_{(I_1,I_2,\bftau)} (P,n):= \sum_{q=1}^\infty \sum_{\substack{r=1\\ (r,q)=1}}^q q^{-|I_3|+|\bftau|}S_{(I_1,I_2,\bftau)}(P;r,q)e\left(-\frac{r}{q}n\right),
\end{equation*}
in case the series is convergent.\par

Recall that $V^*$ denotes the singular locus of the form $F(\bfx)=0$, which is given in affine $s$-space as the zero set of
\begin{equation*}
\frac{\partial F}{\partial x_i}(\bfx)=0,\quad 1\leq i\leq s.
\end{equation*}

We can now state our main theorem.

\begin{theorem}\label{thm1}
Let $d\geq 2$ and $K\geq 1$, and assume that
\begin{equation}\label{eqn0.3}
s-\dim V^*>(d-1)2^{d-1}(2K^2+2K-2).
\end{equation}
Let $\calB$ be given by $\calB= \prod_{i=1}^s (a_i,b_i]$ where $a_i<b_i$ are real numbers for $1\leq i\leq s$. Then we have the asymptotic expansion
\begin{equation*}
\begin{split}
R_\calB(P,n)&= \sum_{(I_1,I_2,\bftau)\in \calI(K)} \grS_{(I_1,I_2,\bftau)}(P,n)\calJ_{(I_1,I_2,\bftau)} (P^{-d}n) P^{s-|I_1|-|I_2|-|\bftau|-d} \\ &+O\left( P^{s-d-(K-1)-\del}\right),
\end{split}
\end{equation*}
for some $\del >0$. Furthermore, all the singular series $\grS_{(I_1,I_2,\bftau)}(P,n)$ and singular integrals $\calJ_{(I_1,I_2,\bftau)} (P^{-d}n)$ occurring in the summation are absolutely convergent.
\end{theorem}

Note that the main term for $I_1=I_2=\emptyset$ and $\bftau=\mathbf{0}$ is excatly the same as the main term in the asymptotic expansion (\ref{eqn0.1}), and $\grS_{(\emptyset,\emptyset,\mathbf{0})}(P,n) =\grS(n)$ and $\grJ_{(\emptyset,\emptyset,\mathbf{0})}(n)=\grJ (n)$ reduce to the classical singular series and singular integral as they appear in (\ref{eqn0.1}).\par
The main new ingredient in the proof of Theorem \ref{thm1} is several applications of Euler-MacLaurin's summation formula in the analysis of the major arcs. These replace the rather crude estimates in the traditional treatment as for example in Lemma 5.1 in \cite{Bir62} and are the key to understanding the major arc contribution in more depth.\par
As pointed out in \cite{VauWooA14}, already in the study of the higher order asymptotic expansion in Waring's problem the singular series occurring in the higher order terms do in general not have an interpretation as an Euler product due to the presence of the Bernoulli polynomials, and hence are difficult to understand. In section \ref{SingSeriesII} we use the multiplication theorem for Bernoulli polynomials to give some interpretation to the singular series $\grS_{(I_1,I_2,\bftau)}(P,n)$, see Lemma \ref{S2} and Lemma \ref{S3}. In particular, we expect that Lemma \ref{S3} turns out to be useful in proving that some of the singular series $\grS_{(I_1,I_2,\bftau)}$ do not vanish.\par
In section \ref{SingIntII} we rewrite the singular integrals $\calJ_{(I_1,I_2,\bftau)} (P^{-d}n)$ in a different way such that we can give a satisfactory interpretation to them. In particular, the singular integral $\calJ_{(I_1,I_2,\bftau)}(n)$ can be viewed as some partial multiple derivative of the function in the variables $x_{i}, i\in I_1\cup I_2$ describing the volume of the bounded piece of the hypersurface $F(\bfx)=n$ inside the box $\prod_{i\in I_3}[a_i,b_i]$, at the point $\bfx_{I_1}=\bfb_{I_1}$ and $\bfx_{I_2}=\bfa_{I_2}$.\par
For $s$ sufficiently large and the special case where $F(\bfx)=\sum_{i=1}^s x_i^d$, we note that the conclusion of Theorem \ref{thm1} reduces to the conclusions of Theorem 1.1 and Theorem 1.2 in \cite{VauWooA14}, and generalizes Theorem 1.2 in \cite{VauWooA14} for the case of odd degree $d$ to an arbitrary number of lower order terms. We compare our results to those obtained in \cite{VauWooA14} in section \ref{comparison}. In particular, we obtain in this way examples stemming from Waring's problem, where one can show that many of the lower order terms actually exist, i.e. are non-zero.\par
\medskip
Alternatively to studying the counting function $R(P,n)$ one could introduce a weighted version of this counting function. If $\ome (\bfx)$ is a smooth and compactly supported weight function and $S_\ome(\alp)= \sum_{\bfx\in \Z^s} \ome(P^{-1}\bfx)e(\alp F(\bfx))$, then this would be given by
\begin{equation*}
R^{(\ome)}(P,n)= \int_0^1 S_\ome(\alp )e(-\alp n)\d\alp.
\end{equation*}
Slight modifications of our proof of Theorem \ref{thm1} show that these techniques establish an asymptotic formula of the form
\begin{equation*}
R^{(\ome)}(P,n)= \grS(n) \calJ(P^{-d}n) P^{s-d} +O\left(P^{-s-d-(K-1)-\del}\right),
\end{equation*}
%\grS_{(\emptyset,\emptyset,\mathbf{0})}(P,n) \grJ_{(\emptyset,\emptyset,\mathbf{0})}(P^{-d}n)
under the assumption that (\ref{eqn0.3}) holds. Hence all the lower order terms for $R^{(\ome)}(P,n)$ vanish identically due to the smooth cut-off function $\ome(\bfx)$.\par

\medskip
Finally, we apply Theorem \ref{thm1} to give higher order expansions for the number of rational points of bounded anticanonical height on certain hypersurfaces of low degree in projective space. In the notation above set $n=0$ and let $F(\bfx)$ be as before a homogeneous polynomial of degree $d$. Then $F(\bfx)=0$ defines a hypersurface $X\subset \P^{s-1}$ of degree $d$. For a rational point $\bfx\in X(\Q)$ given by a representative $\bfx\in \Z^s$ with coprime coordinates $\gcd(x_1,\ldots, x_s)=1$, we define its naive height as
\begin{equation*}
H(\bfx)=\max_{1\leq i\leq s} |x_i|.
\end{equation*}
A central object of study is the counting function 
\begin{equation*}
N_X(P)=\sharp\{ \bfx\in X(\Q): H(\bfx)\leq P\}.
\end{equation*}
Let $\calB=[-1,1]^s$. Via a M\"obius inversion one can express the counting function $N_X(P)$ as
\begin{equation*}
N_X(P)=\frac{1}{2}\sum_{e=1}^\infty \mu(e) \left(R_\calB(e^{-1}P,0)-1\right).
\end{equation*}
Note here that the sum is in fact finite since $R_\calB(e^{-1}P,0)=1$ for $e>P$. In comparison to the usual applications of M\"obius inversion in this setting, we observe that our generalized singular series still depend on $P$. Hence we introduce for $(I_1,I_2,\bftau)\in \calI$ the modified versions 
\begin{equation*}
\widetilde{\grS}_{(I_1,I_2,\bftau)}(P)= \frac{1}{2}\sum_{e=1}^\infty \mu(e) e^{-(s-|I_1|-|I_2|-|\bftau|-d)}\grS_{(I_1,I_2,\bftau)}(e^{-1}P,0),
\end{equation*}
which are absolutely convergent by Lemma \ref{singseries}. As a consequence of Theorem \ref{thm1} we then obtain the following result.

\begin{theorem}\label{thm2}
Let $d\geq 2$ and $K\geq 1$, and assume that (\ref{eqn0.3}) holds. Then one has
\begin{equation*}
\begin{split}
N_X(P)&= \sum_{(I_1,I_2,\bftau)\in \calI(K)} \widetilde{\grS}_{(I_1,I_2,\bftau)}(P)\calJ_{(I_1,I_2,\bftau)} (0) P^{s-|I_1|-|I_2|-|\bftau|-d} \\ &+O\left( P^{s-d-(K-1)-\del}\right),
\end{split}
\end{equation*}
for some $\del >0$.
\end{theorem}

We remark that in the case $n=0$, which in some sense corresponds to Theorem \ref{thm2}, the generalized singular series and singular integrals satisfy some symmetry properties. If $(I_1,I_2,\bftau)\in \calI$ and $(I_1',I_2',\bftau')$ is the dual index tuple given by $I_1'=I_2$ and $I_2'=I_1$ and $\bftau'=\bftau$, then one has 
\begin{equation*}
\calJ_{(I_1',I_2',\bftau')}(0)= (-1)^{|\bftau|}\calJ_{(I_1,I_2,\bftau)}(0),
\end{equation*}
and in the case where $P$ is irrational (or $\tau_i>0$ for all $i\in I_1\cup I_2$) one has
\begin{equation*}
\grS_{(I_1',I_2',\bftau')}(P,0)= (-1)^{|\bftau|}\grS_{(I_1,I_2,\bftau)}(P,0).
\end{equation*}
Hence the corresponding terms in the expansions in Theorem \ref{thm2} give exactly the same contribution for $(I_1,I_2,\bftau)$ and $(I_2,I_1,\bftau)$.\par
It is interesting to view Theorem \ref{thm2} in the light of Manin's conjecture on the number of rational points of bounded anticanonical height on Fano varieties. Note that the lower order terms are of course affected by restricting the counting function to some Zariski-open subset of the original variety. On the other hand, Theorem \ref{thm2} provides a refined asymptotic formula for the number of all rational points of bounded naive or anticanonical height in the case of smooth complete intersections of sufficiently large dimension. In this aspect one should also mention a conjecture of Sir P. Swinnerton-Dyer \cite{SWD}, that in the asymptotic expansion in Manin's conjecture for smooth cubic surfaces one expects a square-root error term with no other lower order terms. It would be interesting to study this question for cubic forms in sufficiently many variables.\par

\medskip
The structure of this paper is as follows. After introducing some notation in the next section, we formulate a multi-dimensional version of Euler-MacLaurin's summation formula in section 3 which is an immediate consequence of the one-dimensional version. After a treatment of the minor arcs in the following section, we give a refined major arc analysis in section 5 based on the use of Euler-MacLaurin's summation formula. 
%This allows us to make the lower order terms visible. 
In section 6 and section 7 respectively, we show that the singular integrals and singular series which we introduced, are absolutely convergent. Together with the previous sections we then deduce the two main Theorems \ref{thm1} and \ref{thm2} in section 8. Section 9 and section 10 contain some finer analysis of the singular integrals and singular series, including some interpretations to all of these objects. Finally, we compare in the last section our main theorem \ref{thm1} to the work of Vaughan and Wooley \cite{VauWooA14} on Waring's problem.
%tand show how our general result reduces to their multi-term asymptotic formula in this special case. 

\medskip

\textbf{Acknowledgements.} The author would like to thank Prof. T. D. Browning for comments on an earlier version of this paper.

\section{Notation and preliminaries}
As usual, we write $\Vert \alp\Vert = \min_{a\in \Z}|a-\alp|$ for the minimal distance from a real number $\alp$ to the next integer. For $x\in \R$ we let $\lfloor x\rfloor$ be the greatest integer which is not larger than $x$, and set $\{x\}= x-\lfloor x\rfloor$. If $\bfa$ and $\bfb$ are real-valued $s$-dimensional vectors than we write $\bfa\leq \bfb$ (or $\bfa<\bfb$) if $a_i\leq b_i$ (or $a_i<b_i$) for all $1\leq i\leq s$.\par
If $I=\{i_1,\ldots, i_l\}$ is a finite index set, then we write $\d \bfx_I$ for $\d x_{i_1}\d x_{i_2}\ldots \d x_{i_l}$ and $|I|$ for the cardinality of $I$.\par
We will often need mixed partial derivatives of functions in several variables. For a multi-index $\bfkap= (\kap_1,\ldots, \kap_s)$ of nonnegative integers we hence introduce the notation
\begin{equation*}
\partial^{\bfkap}_\bfx = \frac{\partial^{\kap_1}}{\partial x_1^{\kap_1}}\ldots\frac{\partial^{\kap_s}}{\partial x_s^{\kap_s}},
\end{equation*}
for this differential operator. Furthermore, we write $|\bfkap|= \sum_{i=1}^s \kap_i$ for the weight of the multi-index $\bfkap$.\par
In Vinogradov's notation all implicit constants may depend in $\bfa,\bfb$ and $F$, and as usual we write $e(x)$ for $e^{2\pi ix}$.\par

\begin{lemma}\label{lemP1}
Let $F(\bfx)$ be a homogeneous polynomial of degree $d$ and $\bfkap\in \Z^s$ a tuple of nonnegative integers such that $\kap_j\geq 1 $ for at least one index $1\leq j\leq s$. Then one has
\begin{equation*}
\partial_\bfx^{\bfkap} e(\gam F(\bfx))= \sum_{a=1}^{|\bfkap|} \gam^a h_a^{(\bfkap)} (\bfx) e(\gam F(\bfx)),
\end{equation*}
where $h_a^{(\bfkap)}$ are homogeneous polynomials in $\bfx$, which are identically zero or of degree $ad-|\bfkap|$.
% Here we understand the degree of the zero polynomial as $-\infty$.
\end{lemma}

\begin{proof}
We prove the lemma by induction on $|\bfkap|$. First assume that $\kap_j=1$ for one $1\leq j\leq s$ and $\kap_i=0$ for $i\neq j$. Then we can directly compute
\begin{equation*}
\partial_\bfx^\bfkap e(\gam F(\bfx))= 2\pi i\gam \partial_{x_j}F(\bfx) e(\gam F(\bfx)).
\end{equation*}
This coincides with the assertion of the lemma since $\partial_{x_j}F(\bfx)$ is a homogeneous polynomial of degree $d-1$ or is identically zero.\par
Next suppose we are given the statement of the lemma for some $\bfkap$. Choose one index $1\leq j\leq s$ and set $\kap_j'=\kap_j+1$ and $\kap_i'=\kap_i$ for $i\neq j$. We now aim to prove the lemma for $\bfkap'$. For this we note that
\begin{equation*}
\partial_\bfx^{\bfkap'} e(\gam F(\bfx))= \partial_{x_j}(\partial_\bfx^{\bfkap}F(\bfx)).
\end{equation*}
By assumption this expression equals
\begin{equation*}
\begin{split}
&\partial_{x_j} \left[ \sum_{a=1}^{|\bfkap|} \gam^a h_a^{(\bfkap)} (\bfx) e(\gam F(\bfx))\right] \\ &= \sum_{a=1}^{|\bfkap|} \gam^a \partial_{x_j}(h_a^{(\bfkap)} (\bfx)) e(\gam F(\bfx)) + \sum_{a=1}^{|\bfkap|} \gam^a h_a^{(\bfkap)} (\bfx) \partial_{x_j}e(\gam F(\bfx)) \\ 
& = \sum_{a=1}^{|\bfkap|} \gam^a (\partial_{x_j}h_a^{(\bfkap)} )(\bfx) e(\gam F(\bfx)) + \sum_{a=1}^{|\bfkap|} 2\pi i\gam^{a+1} h_a^{(\bfkap)} (\bfx) (\partial_{x_j} F) (\bfx) e(\gam F(\bfx)).
\end{split}
\end{equation*}
We note that the degree of the polynomial $(\partial_{x_j}h_a^{(\bfkap)} )(\bfx)$ is $\deg h_a^{(\bfkap)}-1 = da-(\sum_{i=1}^s\kap_i)-1 = da - \sum_{i=1}^s \kap_i'$, if it is non-zero. We next consider the second term in the above expression. We rewrite it as 
\begin{equation*}
\sum_{a=2}^{|\bfkap|+1} \gam^{a} h_{a-1}^{(\bfkap)} (\bfx) (\partial_{x_j} F) (\bfx) e(\gam F(\bfx)).
\end{equation*}
For some $2\leq a \leq |\bfkap'|$ we again note that the degree of the homogeneous polynomial
\begin{equation*}
h_{a-1}^{(\bfkap)} (\bfx) (\partial_{x_j} F) (\bfx) 
\end{equation*}
is equal to $d(a-1) -\sum_{i=1}^s \kap_i +d-1 = da-\sum_{i=1}^s \kap_i'$. This completes the proof of the lemma.
\end{proof}

\section{Euler-MacLaurin summation formula}

We recall the definition of Bernoulli polynomials. The sequence of Bernoulli numbers $B_\kap$ for $\kap\geq 0$ can be defined by setting $B_0=1$ and 
\begin{equation*}
B_\kap= -\sum_{j=0}^{\kap-1} \binom{\kap}{j}\frac{B_j}{\kap-j+1},
\end{equation*}
for $\kap\geq 1$. Then the Bernoulli polynomials $B_\kap(x)$ are given for $\kap\geq 0$ by the formula
\begin{equation*}
B_\kap (x)=\sum_{j=0}^\kap \binom{\kap}{j}B_{\kap-j}x^j.
\end{equation*}
In the following we use the periodic Bernoulli polynomials which are defined as $\bet_\kap(x)=B_\kap(\{x\})$ for $\kap\geq 0$.\par

In our major arc analysis we need a higher dimensional version of the Euler-MacLaurin summation formula which we obtain in successively applying the one-dimensional version. For convenience we hence first state the classical version as for example formulated in Lemma 4.1 in \cite{VauWooA14}.

\begin{lemma}\label{EMS1}
Let $a$ and $b$ be real numbers with $a<b$ and $K$ be a positive integer. Suppose that $G(x)$ has continuous derivatives up to $(K-1)$st order on the interval $[a,b]$, that the $K$-th derivative of $G$ exists and is continuous on $(a,b)$, and that $|G^{(K)}(x)|$ is integrable on $[a,b]$. Then one has
\begin{align*}
\sum_{a<x\leq b}G(x)&=\int_a^b G(x)\d x +\sum_{\kap=1}^K \frac{(-1)^\kap}{\kap!}(\bet_\kap(b)G^{(\kap-1)}(b) -\bet_\kap(a)G^{(\kap-1)}(a))\\ &-\frac{(-1)^K}{K!}\int_{a}^b\bet_K(x)G^{(K)}(x)\d x.
\end{align*}
\end{lemma}

In the work of Vaughan and Wooley \cite{VauWooA14} it is sufficient to use this version of Euler-MacLaurin's summation formula since the diagonal structure of the form underlying Waring's problem ensures that the exponential sum on the major arcs factorizes into one-dimensional sums. We next state a version of Euler-MacLaurin's summation formula which applies to higher dimensional functions, which we obtain from applying the one-dimensional version in each coordinate direction. Since we are only going to apply the Lemma to rather easy and well-behaved functions we do not aim for the greatest generality in the assumptions under which this higher dimensional version of Euler-MacLaurin's summation formula holds. 

\begin{lemma}\label{EMS2}
Assume that $s\geq 1$ and $a_i<b_i$ are real numbers for $1\leq i\leq s$ and let $K_i$ be positive integers for $1\leq i\leq s$. Assume that $g(\bfx)$ has continuous mixed partial derivatives of total order up to $\sum_{i=1}^s K_i$ on the cube $\prod_{i=1}^s [a_i,b_i]$. Then
\begin{align*}
&\sum_{\bfa<\bfx\leq \bfb} g(\bfx) = \sum_{\cup_{i=1}^4 I_i=\{1,\ldots, s\}} \left(\prod_{i\in I_4} \frac{(-1)^{K_i+1}}{K_i!}\right)\int_{\prod_{i\in I_3\cup I_4}[a_i,b_i]}\left( \prod_{i\in I_4} \bet_{K_i}(x_i)\right) \\ &\quad\quad\quad\quad\quad
\left[ \prod_{i\in I_1} \left( \sum_{\kap_i=1}^{K_i} \frac{(-1)^{\kap_i}}{\kap_i!} \bet_{\kap_i}(b_i) \left( \frac{\partial}{\partial x_i}\right)^{(\kap_i-1)}\right)
\right.  \\ &\left.\prod_{i\in I_2} \left(\sum_{\kap_i=1}^{K_i} \frac{(-1)^{\kap_i+1}}{\kap_i!} \bet_{\kap_i}(a_i) \left(\frac{\partial}{\partial x_i}\right)^{(\kap_i-1)}\right)  \prod_{i\in I_4} \left(\frac{\partial}{\partial x_i}\right)^{K_i} g (\bfx)
\right]_{\bfx=\sig_{\bfa,\bfb}(\bfx)}\d \bfx_{I_3} \d \bfx_{I_4}
\end{align*}
The summation over $\cup_{i=1}^4 I_i$ is over all possible partitions of $\{1,\ldots, s\}$ into four disjoint index sets $I_i$, $1\leq i\leq 4$. 
\end{lemma}

\section{Minor arc estimates}
In this section we assume that $F(\bfx)$ is a polynomial in $\bfx$ of degree $d$, not necessarily homogeneous. We define the exponential sum
$$ S(\alp)= \sum_{\bfx\in P\calB} e(\alp F(\bfx)).$$
By orthogonality we can express the counting function $R_\calB(P,n)$ as
\begin{equation*}
R_\calB(P,n)= \int_0^1 S(\alp)e(-\alp n)\d\alp.
\end{equation*}
In order to apply the circle method to this counting function we need to dissect the unit interval $[0,1]$ into major and minor arcs. This is done in a traditional way following for example the work of Birch \cite{Bir62}.

Let $0<\eta<1/2$ be some small parameter to be chosen later. For coprime integers $r,q$ we define the major arc
\begin{equation*}
\grM_{r,q}'(\eta) = \{ \alp\in [0,1): |q\alp -r|\leq q P^{-d+\eta}\},
\end{equation*}
and the major arcs $\grM'(\eta)$ as the union
\begin{equation}\label{eqn4.2}
\grM'(\eta)= \bigcup_{1\leq q\leq P^\eta} \bigcup_{\substack{1\leq r\leq q\\ (r,q)=1}} \grM_{r,q}'(\eta).
\end{equation}
Similarly, we define slightly smaller major arcs by 
\begin{equation*}
\grM_{r,q}(\eta) = \{ \alp\in [0,1): |q\alp -r|\leq  P^{-d+\eta}\},
\end{equation*}
for coprime integers $r,q$ and 
\begin{equation*}
\grM(\eta)= \bigcup_{1\leq q\leq P^\eta} \bigcup_{\substack{1\leq r\leq q\\ (r,q)=1}} \grM_{r,q}(\eta).
\end{equation*}
If the parameter $\eta$ is clear from the context we sometimes use the shorter notation $\grM'_{r,q}$ for $\grM'_{r,q}(\eta)$ and for the major arcs $\grM_{r,q}$ similarly.\par
Furthermore, we define the minor arcs as the complement of the smaller version of the major arcs, i.e. $\grm(\eta)=[0,1)\setminus \grM(\eta)$.\par
We recall Lemma 4.1 from Birch's paper \cite{Bir62} which asserts that the major arcs $\grM(\eta)$ are disjoint in case that $\eta$ is sufficiently small. For convenience we state it here for the slightly larger major arcs $\grM'(\eta)$ since we need it for them in the later analysis of major arcs.

\begin{lemma}\label{MAdisj}
Assume that $3\eta <d$. Then the union defining the major arcs $\grM'(\eta)$ as in (\ref{eqn4.2}) is disjoint for $P$ sufficiently large.
\end{lemma}

\begin{proof}
Assume that we are given some $\alp\in \grM'_{r,q}(\eta)\cap \grM'_{r',q'}(\eta)$ for coprime pairs of integers $r,q$ and $r',q'$ with $q,q'\leq P^{\eta}$ such that $r/q\neq r'/q'$. We can estimate
\begin{equation*}
1\leq |r'q-rq'|\leq q'|q\alp-r|+q|q'\alp-r'|\leq 2P^{2\eta}P^{-d+\eta}=2P^{-d+3\eta}.
\end{equation*}
This is a contradiction if we take $P$ sufficiently large depending on $d-3\eta$.
\end{proof}

It is convenient to introduce more generally for any function $\ome: \R^s\rightarrow \R$ of compact support the exponential sum
$$S_\ome (\alp,P)= \sum_{\bfx\in \Z^s} \ome \left(\frac{\bfx}{P}\right) e(\alp F(\bfx)).$$
Hence if we take $\ome$ to be the indicator function of the box $\calB$, then we recover $S(\alp)$. For a certain class of weight functions $\ome$ we need a form of Weyl's inequality for $S(\alp)$. For this we use a slight modification of recent work of Browning and Prendiville \cite{BroPreA14}.\par
We first recall some conventions from \cite{BroPreA14}. We say that a pair $\alp\in \R/\Z$ and $q\in \N$ is primitive, if there is some $r\in \Z$ with $(r,q)=1$ and $\Vert q\alp\Vert = |q\alp -r|$. For some positive constants $c$, $C$ and a positive integer $m$ we introduce the class of smooth weight functions $\calS (c,C,m)$ as the set of smooth compactly supported functions $\ome : \R^s\rightarrow [0,\infty)$ such that $\supp (\ome)\subset [-c,c]^s$ and $\Vert\partial_\bfx^\bfkap \ome (\bfx)\Vert_\infty \leq C$ for all multi-indices $\bfkap \in (\N\cup \{0\})^s$ with $|\bfkap|\leq m$.\par
If $F(\bfx)$ is a polynomial in $\bfx$, then we write $F^{[d]}(\bfx)$ for its homogeneous part of degree $d$. Furthermore, we write $\Sing(F^{[d]})$ for the singular locus of the affine variety given by $F^{[d]}(\bfx)=0$, which is the zero locus of the system of equations
\begin{equation*}
\frac{\partial F^{[d]}}{\partial x_i}(\bfx)=0,\quad 1\leq i\leq s.
\end{equation*}

\begin{lemma}\label{Weyl}
Assume that $\alp$ and $q$ are primitive. Let $\ome \in \calS(c,C,m)$ and $\chi$ the indicator function of some box in $\R^s$, which is contained in $[-c,c]^s$. Assume that $m\geq s$. Then one has
\begin{equation*}
\left| \frac{S_{\ome\chi} (\alp,P)}{P^s}\right|^{2^{d-1}} \ll_{c,C,m} (\log P)^s \left( P^{1-d}+\Vert q\alp\Vert +q P^{-d} +\min\left\{q^{-1},\frac{1}{\Vert q\alp \Vert P^d}\right\}\right)^{\frac{s-\sig}{d-1}},
\end{equation*}
where $\sig = \dim \Sing (F^{[d]})$ is the dimension of the singular locus of the affine variety given by $F^{[d]}(\bfx)=0$.
\end{lemma}

This is a consequence of Lemma 3.3 in \cite{BroPreA14}. Note that we do not need an explicit dependence on the bound for $S_\ome (\alp,P)$ depending on the coefficients of $F(\bfx)$ which can be found in the formulation of Lemma 3.3 in \cite{BroPreA14}. 

\begin{proof}
The proof is identical to the proof of Lemma 3.3 in \cite{BroPreA14}, with the function $\phi(\bfx)= e(\alp F_{\bfh_1,\ldots, \bfh_{d-1}}(\bfx))$ in the notation of \cite{BroPreA14} replaced by the product $\phi(\bfx)= \chi_{(\bfh_1,\ldots, \bfh_{d-1})/P}(\bfx/P)e(\alp F_{\bfh_1,\ldots, \bfh_{d-1}}(\bfx))$. For this we note that $\chi_{(\bfh_1,\ldots, \bfh_{d-1})/P}(\bfx/P)$ is again the indicator function of a box, since the intersection of two boxes in $\R^s$ is again a box.
\end{proof}

As a first application of Weyl's Lemma \ref{Weyl} we provide an upper bound for the minor arc contribution to $R_\calB (P,n)$. In the following we assume that $F(\bfx)$ is homogeneous and set $\sig = \dim \Sing (F)$ as in Lemma \ref{Weyl}.

\begin{lemma}\label{minarcs}
Let $0<\tet_0<(1/2) d$ and assume that $s-\sig > (d-1)2^d$. Then one has 
\begin{equation*}
\int_{\alp\notin \grM(\tet_0)}|S(\alp)|\d\alp = O_\del\left(P^{s-d+\del-\tet_0\left(2^{-d+1}\frac{s-\sig}{d-1}-2\right)}\right),
\end{equation*}
for any $\del >0$.
\end{lemma}

\begin{proof}
Let $Q\geq 1$. Dirichlet's approximation lemma states that for any real $\alp\in\R$ there is some $1\leq q\leq Q$ with $\Vert \alp q\Vert \leq Q^{-1}$. We claim that one might additionally assume that the pair $(\alp,q)$ is primitve. Otherwise, there is some $r\in\Z$ such that $|\alp q-r|=\Vert \alp q\Vert$ with $\gcd (r,q)>1$. Then there is some divisor $q'|q$ such that $\alp,q'$ is primitive and $\Vert \alp q'\Vert <\Vert \alp q\Vert \leq Q^{-1}$.\par
Assume now that $\alp \in \grm (\tet)$ for some $0<\tet\leq (1/2)d$. Then there is some $q\leq Q$ such that $\alp,q$ is primitive and $\Vert \alp q\Vert \leq Q^{-1}$. Furthermore one has $q>P^\tet$ or $\Vert \alp q\Vert > P^{-d+\tet}$, since otherwise $\alp$ would be contained in the major arcs $\grM(\tet)$. We now apply the Weyl bound in Lemma \ref{Weyl} to the exponential sum $S(\alp)$, where we set $\chi$ the characteristic function of the box $\calB$ and $\ome$ a smooth function such that $\ome \equiv 1 $ on the box $\calB$. Then Lemma \ref{Weyl} delivers the bound
\begin{equation*}
\left| \frac{S(\alp)}{P^s}\right|^{2^{d-1}}\ll (\log P)^s \left( P^{1-d} + \frac{1}{Q} + QP^{-d} + \min \left\{ q^{-1},\frac{1}{\Vert q\alp\Vert P^d}\right\} \right)^{\frac{s-\sig}{d-1}}.
\end{equation*}
Note that $\min \left\{ q^{-1}, \frac{1}{\Vert q\alp\Vert P^d}\right\} \leq P^{-\tet}$ and set $Q=P^\tet$. Then we obtain
\begin{equation*}
\left| \frac{S(\alp)}{P^s}\right|^{2^{d-1}}\ll (\log P)^s \left(P^{1-d}+P^{-\tet}+P^{\tet-d}\right)^{\frac{s-\sig}{d-1}}.
\end{equation*}
Note that our restriction $0<\tet\leq (1/2)d$ implies that the second term in that bound dominates the expression, i.e.
\begin{equation}\label{star1}
\left| \frac{S(\alp)}{P^s}\right|^{2^{d-1}}\ll (\log P)^s (P^{-\tet})^{\frac{s-\sig}{d-1}}.
\end{equation}
Now we define a sequence 
$$ \tet_T>\tet_{T-1}>\ldots > \tet_1>\tet_0>0,$$
such that $\tet_T= (1/2)d$ and $|\tet_{t+1}-\tet_t|\leq \del$ for all $1\leq t<T$. We can do this with at most $T\ll P^\del$ points. Note that by Dirichlet's approximation theorem we have $\grM(\tet_T)=[0,1)$. We now estimate the contribution of 
\begin{equation}\label{star2}
\int_{\alp \in \grM(\tet_{t+1})\setminus \grM(\tet_t)}|S(\alp)|\d\alp,
\end{equation}
for all $0\leq t <T$. By the Weyl bound (\ref{star1}) we obtain
\begin{equation*}
\int_{\alp \in \grM(\tet_{t+1})\setminus \grM(\tet_t)}|S(\alp)|\d\alp\ll \meas \left(\grM(\tet_{t+1})\right) P^{s+\del} P^{-2^{-d+1}\tet_t\frac{s-\sig}{d-1}}.
\end{equation*}
Note that the major arcs $\grM(\tet_{t+1})$ might not be disjoint, but we can still bound their measure above by
\begin{equation*}
\meas \left(\grM(\tet_{t+1})\right) \ll \sum_{q\leq P^{\tet_{t+1}}}\sum_{r=1}^q q^{-1}P^{-d+\tet_{t+1}} \ll P^{2\tet_{t+1}-d}.
\end{equation*}
Hence we may bound the contribution of (\ref{star2}) by
\begin{equation*}
\ll P^{2\tet_{t+1}-d+s+\del-\tet_t 2^{-d+1}\frac{s-\sig}{d-1}}\ll P^{s-d+3\del-\tet_t (2^{-d+1}\frac{s-\sig}{d-1}-2)}.
\end{equation*}
Hence we can bound the complete minor arc contribution by
\begin{equation*}
\int_{\alp\notin \grM(\tet_0)}|S(\alp)|\d\alp \ll \sum_{t=0}^{T-1} \int_{\alp \in \grM(\tet_{t+1})\setminus \grM(\tet_t)}|S(\alp)|\d\alp\ll P^{s-d+4\del-\tet_0 (2^{-d+1}\frac{s-\sig}{d-1}-2)},
\end{equation*}
which completes the proof of the lemma.
\end{proof}

\section{Major arc analysis}

The main goal of this section is to replace the usual major arc approximation as for example in Lemma 5.1 in Birch's work \cite{Bir62} by a much finer approximation using the higher dimensional version of Euler-MacLaurin's summation formula in Lemma \ref{EMS2}. For this recall the notation of the singular series and singular integrals as in the introduction as well as the integrals $J_{(I_1,I_2,\bftau)}(\gam)$ and the exponential sums $S_{(I_1,I_2,\bftau)}(P;r,q)$. In addition, we define the function $f(\gam,\bfx)$ by
\begin{equation*}
f(\gam,\bfx)= e(\gam F(\bfx)),
\end{equation*}
and write $f^{(\bfkap)}(\gam,\bfx):= \partial_\bfx^{\bfkap} f(\gam,\bfx)$.\par

We are now in a position to state our first major arc approximation to the exponential sum $S(\alp)$.

\begin{lemma}\label{MA1}
Assume that $\alp\in \grM'_{r,q}$ for some $q\leq P^\eta$, and write $\alp= \frac{r}{q} +\gam$ with some $|\gam|\leq P^{-d+\eta}$. Let $K\geq 1$ be an integer. Then we have
\begin{equation*}
\begin{split}
&S(\alp)e(-\alp n) \\ &= \sum_{(I_1,I_2,\bftau)\in \calI (K)} q^{-|I_3|+|\bftau|} S_{(I_1,I_2,\bftau)}(P;r,q) e\left(-\frac{r}{q}n\right) P^{|I_3|-|\bftau|}J_{(I_1,I_2,\bftau)}(P^d\gam)e(-\gam n) \\ &+O(P^{s-K+2K\eta}).
\end{split}
\end{equation*}
\end{lemma}

Note that the term for $I_1=I_2=\emptyset $ and $\bftau=0$ corresponds to the usual approximation on the major arcs as for example in Lemma 5.1 in Birch's work \cite{Bir62}. All the other terms will contribute to lower order terms.

\begin{proof}
We start in rewriting the exponential sum $S(\alp)$ as
\begin{equation}\label{eqnMAlem1}
S(\alp)= \sum_{0\leq \bfz <q} e\left(\frac{r}{q} F(\bfz)\right) \sum_{\bfz+q\bfy \in P\calB} e(\gam F(\bfz+q\bfy)),
\end{equation}
and consider the inner sum for a fixed vector $\bfz$. Let $\atil_i,\btil_i$ for $1\leq i\leq s$ be defined by
\begin{equation*}
\prod_{i=1}^s (\atil_i,\btil_i]= \prod_{i=1}^s \left( \frac{Pa_i-z_i}{q},\frac{Pb_i-z_i}{q}\right].
\end{equation*}
%i.e. we have $\atil_i= \frac{Pa_i-z_i}{q}$ and $\btil_i= \frac{Pb_i-z_i}{q}$ for all $1\leq i\leq s$. 
Let $g(\bfy)= f(\gam,\bfz+q\bfy)$ and note that 
\begin{equation}\label{eqnMAlem1b}
\partial_\bfy^{\bfkap}g(\bfy)= q^{|\bfkap|} f^{(\bfkap)} (\gam,\bfz+q\bfy),
\end{equation}
for every multi-index $\bfkap\in \Z_{\geq 0}^s$.\par
Now choose some fixed $K\in \N$ and let $\calItil (K)$ be the set of tuples $(I_1,I_2,I_4,\bftau)$ with the following properties. For $i=1,2,4$ one has $I_i\subset \{1,\ldots, s\}$ and the index sets $I_i$ are pairwise disjoint. Furthermore $\bftau\in \Z_{\geq 0}^s$ satisfies $\tau_i=0$ if $i\notin I_1\cup I_2\cup I_4$, and $0\leq \tau_i\leq K-1$ for $i\in I_1\cup I_2$ and $\tau_i=K$ for $i\in I_4$. Similarly as before we set $I_3 = \{1,\ldots, s\}\setminus (I_1\cup I_2\cup I_4)$ for such a tuple in $\calItil (K)$. Now we apply Lemma \ref{EMS2} to the sum
\begin{equation}\label{eqnMAlem2}
\Sig(\bfz)=\sum_{\bfz+q\bfy\in P\calB} e(\gam F(\bfz+q\bfy)),
\end{equation}
with the parameters $K_i=K$ for all $1\leq i\leq s$ and to the box $\prod_{i=1}^s (\atil_i,\btil_i]$. We obtain
\begin{equation}\label{eqnMAlem3}
\Sig(\bfz)= \sum_{(I_1,I_2,I_4,\bftau)\in \calItil (K)}  \Sig(\bfz;I_1,I_2,I_4,\bftau),
\end{equation}
with
\begin{equation*}
\begin{split}
&\Sig(\bfz;I_1,I_2,I_4,\bftau)= \left(\prod_{i\in I_4} \frac{(-1)^{K+1}}{K!}\right) \left(\prod_{i\in I_1} \frac{(-1)^{\tau_i+1}}{(\tau_i+1)!} \bet_{\tau_i+1}(\btil_i)\right)\\ & \left(\prod_{i\in I_2} \frac{(-1)^{\tau_i}}{(\tau_i+1)!} \bet_{\tau_i+1}(\atil_i)\right) 
\int_{\prod_{i\in I_3\cup I_4}[\atil_i,\btil_i]} \left(\prod_{i\in I_4} \bet_K(x_i)\right) \partial_\bfx^{\bftau}g(\bfx)|_{\bfx= \sig_{\bfatil,\bfbtil}(\bfx)} \d \bfx_{I_3} \d \bfx_{I_4}.
\end{split}
\end{equation*}
First we estimate the contribution of $\Sig (\bfz;I_1,I_2,I_4,\bftau)$ in the case $|I_1|+|I_2|+|\bftau|\geq K$. We apply Lemma \ref{lemP1} and obtain
\begin{equation}\label{eqnMAlem3b}
f^{(\bftau)} (\gam,\bfx)= \sum_{l=1}^{|\bftau|}\gam^l h_l^{(\bftau)}(\bfx) e(\gam F(\bfx)),
\end{equation}
with $h_l^{(\bftau)}(\bfx)$ some homogeneous polynomials in $\bfx$ which are either zero or of degree $ld-|\bftau|$. Hence we have
\begin{equation*}
\partial_\bfy^{(\bftau)} g(\bfy)= q^{|\bftau|} \sum_{l=1}^{|\bftau|} \gam^l h_l^{(\bftau)} (\bfz+q\bfy) e(\gam F(\bfz+q\bfy)).
\end{equation*}
For a point $\bfy$ lying in the box $\bfy\in \prod_{i=1}^s[\atil_i,\btil_i]$ we can now estimate
\begin{equation*}
|\partial_\bfy^{(\bftau)}g(\bfy)| \ll q^{|\bftau|} \sum_{l=1}^{|\bftau|} |\gam|^l P^{\deg (h_l^{(\bftau)})}\ll q^{|\bftau|} \sum_{l=1}^{|\bftau|} |\gam|^l P^{ld-|\bftau|}.
\end{equation*}
Note that all the periodic Bernoulli polynomials $\bet_\tau(x)$ are bounded. Hence we can now estimate $\Sig(\bfz;I_1,I_2,I_4,\bftau)$ by
\begin{equation*}
\Sig (\bfz;I_1,I_2,I_4;\bftau)\ll \int_{\prod_{i\in I_3\cup I_4} [\atil_i,\btil_i]} q^{|\bftau|} \sum_{l=1}^{|\bftau|} |\gam|^l P^{ld-|\bftau|} \d \bfx_{I_3} \d \bfx_{I_4}.
\end{equation*}
Since the volume of $\prod_{i\in I_3 \cup I_4}[\atil_i,\btil_i]$ is bounded by $\left(\frac{P}{q}\right)^{|I_3|+|I_4|}$ we obtain the upper bound
\begin{equation*}
\begin{split}
\Sig (\bfz;I_1,I_2,I_4;\bftau)&\ll \left(\frac{P}{q}\right)^{|I_3|+|I_4|}q^{|\bftau|} \sum_{l=1}^{|\bftau|} |\gam|^l P^{ld-|\bftau|} \\
&\ll P^{|I_3|+|I_4|} q^{|\bftau|-|I_3|-|I_4|} \sum_{l=1}^{|\bftau|} P^{l(-d+\eta)}P^{ld-|\bftau|}\\ & \ll P^{|I_3|+|I_4|} q^{|\bftau|-|I_3|-|I_4|} P^{|\bftau|\eta-|\bftau|}
\end{split}
\end{equation*}
We estimate this further as
\begin{equation}\label{eqnMAlem4}
\begin{split}
\Sig (\bfz;I_1,I_2,I_4;\bftau)&\ll q^{-s} P^{s-|I_1|-|I_2|-|\bftau|+|\bftau|\eta}q^{|\bftau|+|I_1|+|I_2|} \\ & \ll q^{-s} P^{s-|I_1|-|I_2|-|\bftau|}P^{2\eta |\bftau| +\eta (|I_1|+|I_2|)}\\ &\ll q^{-s} P^{s-(1-2\eta)(|I_1|+|I_2|+|\bftau|)} \\ &\ll q^{-s} P^{s-K+2 \eta K},
\end{split}
\end{equation}
using $\eta <1/2$.\par
We combine this information with equations (\ref{eqnMAlem1}), (\ref{eqnMAlem2}), (\ref{eqnMAlem3}), (\ref{eqnMAlem4}), and see that
\begin{equation}\label{eqnMAlem5}
S(\alp)= \sum_{(I_1,I_2,\bftau)\in \calI(K)}\sum_{0\leq \bfz<q} e\left(\frac{r}{q}F(\bfz)\right) \Sig(\bfz;I_1,I_2,\emptyset,\bftau)+O(P^{s-K+2\eta K}).
\end{equation}
Note for this that $|\bftau|\geq K$ as soon as $I_4\neq \emptyset$.\par
Next we consider for a fixed tuple $(I_1,I_2,\bftau)\in \calI(K)$ the integral
\begin{equation*}
J'= \int_{\prod_{i\in I_3}[\atil_i,\btil_i]} \partial_\bfy^{\bftau}g(\bfy)|_{\bfy=\sig_{\bfatil,\bfbtil}}(\bfy) \d \bfy_{I_3}.
\end{equation*}
We recall the relation (\ref{eqnMAlem1b}) and perform the variable substitution $x_i=qy_i+z_i$ for $i\in I_3$. This leads us to
\begin{equation*}
\begin{split}
J'&= q^{-|I_3|} \int_{\prod_{i\in I_3}[Pa_i,Pb_i]}q^{|\bftau|} f^{(\bftau)}(\gam,\bfz+q\bfy)|_{\bfy= \sig_{\bfatil,\bfbtil}(\frac{\bfx-\bfz}{q})}\d \bfx_{I_3} \\ 
&= q^{-|I_3|+|\bftau|} \int_{\prod_{i\in I_3}[Pa_i,Pb_i]} f^{(\bftau)}(\gam,\bfx)|_{\bfy= \sig_{P\bfa,P\bfb}(\bfx)}\d \bfx_{I_3}\\
&= q^{-|I_3|+|\bftau|} \int_{\prod_{i\in I_3}[Pa_i,Pb_i]} f^{(\bftau)}(\gam,\sig_{P\bfa,P\bfb}(\bfx))\d \bfx_{I_3}.
\end{split}
\end{equation*}
Note that $J'$ is now independent of $\bfz$. We further rewrite $J'$ via substituting $Px_i'=x_i$ for $i\in I_3$, and obtain
\begin{equation*}
J'= q^{-|I_3|+|\bftau|}P^{|I_3|} \int_{\prod_{i\in I_3}[a_i,b_i]} f^{(\bftau)}(\gam,\sig_{P\bfa,P\bfb}(P\bfx))\d \bfx_{I_3}.
\end{equation*}
We recall equation (\ref{eqnMAlem3b}) and observe that
\begin{equation*}
\begin{split}
f^{(\bftau)}(\gam,P\bfx)&=\sum_{l=1}^{|\bftau|}\gam^l h_l^{(\bftau)}(P\bfx) e(\gam F(P\bfx))\\ 
&= \sum_{l=1}^{|\bftau|}\gam^l P^{ld-|\bftau|}h_l^{(\bftau)}(\bfx) e(\gam P^d F(\bfx))\\ 
&= P^{-|\bftau|}\sum_{l=1}^{|\bftau|}(P^d\gam)^l h_l^{(\bftau)}(\bfx) e(P^d\gam  F(\bfx))\\ 
&= P^{-|\bftau|} f^{(\bftau)} (P^d\gam,\bfx).
\end{split}
\end{equation*}
Hence we can again reformulate $J'$ as
\begin{equation*}
\begin{split}
J'&= q^{-|I_3|+|\bftau|}P^{|I_3|-|\bftau|} \int_{\prod_{i\in I_3}[a_i,b_i]} f^{(\bftau)}(P^d\gam,\sig_{\bfa,\bfb}(\bfx))\d \bfx_{I_3}\\ &=q^{-|I_3|+|\bftau|}P^{|I_3|-|\bftau|} J_{(I_1,I_2,\bftau)}(P^d\gam).
\end{split}
\end{equation*}
We conclude that 
\begin{equation*}
\begin{split}
\Sig(\bfz;I_1,I_2,\emptyset,\bftau)= &\left(\prod_{i\in I_1} \frac{(-1)^{\tau_i+1}}{(\tau_i+1)!} \bet_{\tau_i+1}(\btil_i)\right) \left(\prod_{i\in I_2} \frac{(-1)^{\tau_i}}{(\tau_i+1)!} \bet_{\tau_i+1}(\atil_i)\right)J' \\ = &\left(\prod_{i\in I_1} \frac{(-1)^{\tau_i+1}}{(\tau_i+1)!} \bet_{\tau_i+1}(\btil_i)\right) \left(\prod_{i\in I_2} \frac{(-1)^{\tau_i}}{(\tau_i+1)!} \bet_{\tau_i+1}(\atil_i)\right)\\ &
q^{-|I_3|+|\bftau|}P^{|I_3|-|\bftau|} J_{(I_1,I_2,\bftau)}(P^d\gam).
\end{split}
\end{equation*}
The Lemma now follows from inserting this into equation (\ref{eqnMAlem5}).

\end{proof}

We now use this Lemma to evaluate the major arc contribution to the counting function $R(P,n)$. For a measurable subset $\calC \subset [0,1]$ we write
\begin{equation*}
R(P,n;\calC)= \int_\calC S(\alp )e(-\alp n) \d \alp.
\end{equation*}
Hence our next goal is to further analyse $R(P,n;\grM')$. For a tuple $(I_1,I_2,\bftau)\in \calI$, we introduce the truncated singular series
\begin{equation*}
\grS_{(I_1,I_2,\bftau)} (P,n;Q):= \sum_{q\leq Q} \sum_{\substack{r=1\\ (r,q)=1}}^q q^{-|I_3|+|\bftau|}S_{(I_1,I_2,\bftau)}(P;r,q)e\left(-\frac{r}{q}n\right).
\end{equation*}
We write
\begin{equation*}
\grS_{(I_1,I_2,\bftau)} (P,n) = \lim_{Q\rightarrow \infty} \grS_{(I_1,I_2,\bftau)} (P,n;Q),
\end{equation*}
if the limit exists.
Similarly, for any real number $Q\geq 1$ and $(I_1,I_2,\bftau)\in \calI$, we introduce the truncated singular integral
\begin{equation*}
\calJ_{(I_1,I_2,\bftau)}(n;Q):= \int_{|\gam|\leq Q} J_{(I_1,I_2,\bftau)}(\gam)e(-\gam n)\d\gam,
\end{equation*}
and we write
\begin{equation*}
\calJ_{(I_1,I_2,\bftau)}(n) =\lim_{Q\rightarrow \infty} \calJ_{(I_1,I_2,\bftau)}(n;Q),
\end{equation*}
in case the integral converges.\par

\begin{lemma}\label{MA2}
Let $K\geq 1$ be some integer and $\eta < (1/3)d$. Then one has
\begin{equation*}
\begin{split}
R(P,n;\grM')&= \sum_{(I_1,I_2,\bftau)\in \calI(K)} \grS_{(I_1,I_2,\bftau)}(P,n;P^\eta)\calJ_{(I_1,I_2,\bftau)} (P^{-d n};P^\eta) P^{|I_3|-|\bftau|-d} \\ &+O\left( P^{s-K-d+(2K+3)\eta}\right).
\end{split}
\end{equation*}
\end{lemma}

\begin{proof}
By definition of the major arcs we have
\begin{equation*}
R(P,n;\grM')= \sum_{1\leq q\leq P^\eta} \sum_{\substack{r=1\\ (r,q)=1}}^q \int_{\grM'_{r,q}}S(\alp) e(-\alp n)\d \alp.
\end{equation*}
We use Lemma \ref{MA1} to approximate $S(\alp)$ on $\grM_{r,q}'$, and obtain
\begin{equation}\label{eqnMA2lem1}
\begin{split}
R(P,n;\grM')=\sum_{(I_1,I_2,\bftau)\in \calI(K)} &\grS_{(I_1,I_2,\bftau)}(P,n;P^\eta) P^{|I_3|-|\bftau|}\\ &\int_{|\gam|\leq P^{-d+\eta}} J_{(I_1,I_2,\bftau)}(P^d\gam) e(-\gam n)\d\gam + E_1,
\end{split}
\end{equation}
with some error term $E_1$. By Lemma \ref{MA1} we can bound the resulting error $E_1$ by
\begin{equation*}
\begin{split}
E_1&\ll \sum_{1\leq q\leq P^\eta} \sum_{r=1}^q \meas (\grM'_{r,q}) P^{s-K+2K\eta} 
\\&\ll P^{2\eta} P^{-d+\eta} P^{s-K+2K\eta}\ll P^{s-K-d+(2K+3)\eta}.
\end{split}
\end{equation*}
Furthermore we note that the variable substitution $\gam'= P^d\gam$ leads in the integral part to
\begin{equation*}
\begin{split}
\int_{|\gam|\leq P^{-d+\eta}}J_{(I_1,I_2,\bftau)}(P^d\gam) e(-\gam n)\d\gam &= P^{-d} \int_{|\gam|\leq P^\eta} J_{(I_1,I_2,\bftau)}(\gam) e(-\gam P^{-d} n)\d\gam \\
&= P^{-d} \calJ_{(I_1,I_2,\bftau)} (P^{-d}n,P^\eta).
\end{split}
\end{equation*}
Together with equation (\ref{eqnMA2lem1}) and the estimate for the error term $E_1$ this completes the proof of the lemma.
\end{proof}

\section{Singular series}
The first goal of this section is to study convergence properties of the truncated singular series $\grS_{(I_1,I_2,\bftau)}(P,n;Q)$. 

\begin{lemma}\label{singseries}
Let $(I_1,I_2,\bftau)\in \calI (K)$, for some $K\geq 1$. Assume that
\begin{equation*}
2^{-d+1}\frac{s-\sig}{d-1}>K+1.
\end{equation*}
Then  $\grS_{(I_1,I_2,\bftau)}(P,n;Q)$ is absolutely convergent and satisfies
\begin{equation*}
\grS_{(I_1,I_2,\bftau)}(P,n;Q)-\grS_{(I_1,I_2,\bftau)}(P,n)\ll_\bftau Q^{K+1-2^{-d+1}\frac{s-\sig}{d-1}+\eps},
\end{equation*}
for any $\eps >0$, and the implicit constant depends on $\bftau$ but not on $P$.
\end{lemma}

\begin{proof}
The main ingredient of the proof is a suitable upper bound for the exponential sum $S_{(I_1,I_2,\bftau)}(P;r,q)$ which we deduce from Lemma \ref{Weyl}. Recall that 
\begin{equation*}
\grS_{(I_1,I_2,\bftau)}(P,n;Q)= \sum_{q\leq Q}\sum_{\substack{r=1\\ (r,q)=1}}^q q^{-|I_3|+|\bftau|}S_{(I_1,I_2,\bftau)}(P;r,q)e\left(-\frac{r}{q}n\right),
\end{equation*}
with exponential sums of the form
\begin{equation*}
S_{(I_1,I_2,\bftau)}(P;r,q)=\sum_{\bfz\in\Z^s} e\left(\frac{r}{q}F(\bfz)\right) h(\bfz/q),
\end{equation*}
where the weight function $h(\bfz)$ is given by
\begin{equation*}
\begin{split}
h(\bfz)=\id_{[0,1)^s} (\bfz)&\left(\prod_{i\in I_1}\frac{(-1)^{\tau_i+1}}{(\tau_i+1)!}\bet_{\tau_i+1}\left(\frac{Pb_i}{q}-z_i\right)\right) \\ &\left(\prod_{i\in I_2} \frac{(-1)^{\tau_i}}{(\tau_i+1)!} \bet_{\tau_i+1}\left(\frac{Pa_i}{q}-z_i\right)\right).
\end{split}
\end{equation*}
Note that each of the $\bet_{\tau_i+1}\left(\frac{Pa_i}{q}-z_i\right)$ is a polynomial of bounded degree depending only on $\tau_i$. Hence one can divide the interval $[0,1)$ into at most $1+\deg \bet_{\tau_i+1}$ subintervals, on each of which $\bet_{\tau_i+1}$ does not change sign. We do this for each $i\in I_1\cup I_2$ and obtain a finite set of boxes on each of which $h(\bfz)$ has bounded derivatives up total order at least $s$. Hence we can write
\begin{equation*}
h(\bfz)= \sum_{l=1}^L \chi_l(\bfz) \ome_l(\bfz),
\end{equation*}
where $L\ll_{|\bftau|}1$ and $\chi_l$ is the indicator function of a box contained in $[0,1)^s$ and $\ome_l(\bfz)\in \calS (c,C,s)$ for some positive constants $c$ and $C$. Note that both $c$ and $C$ do not depend on $P$, and $c,C\ll_{|\bftau|} 1$. Hence we can rewrite
\begin{equation*}
S_{(I_1,I_2,\bftau)}(P;r,q)= \sum_{l=1}^L \sum_{\bfz\in \Z^s} \chi_l\left(\frac{\bfz}{q}\right)\ome_l \left(\frac{\bfz}{q}\right) e\left(\frac{r}{q}F(\bfz)\right).
\end{equation*}
We now apply Lemma \ref{Weyl} to each of the inner exponential sums. Note that if $(r,q)=1$, then the tuple $r/q$, $q$ is primitive. 
%since $|q\cdot (r/q)-r|= \Vert q r/q\Vert =0$ and $(r,q)=1$.
Hence we obtain
\begin{equation*}
\left|q^{-s}\sum_{\bfz\in \Z^s} \chi_l\left(\frac{\bfz}{q}\right)\ome_l\left(\frac{\bfz}{q}\right) e\left(\frac{r}{q}F(\bfz)\right)\right|^{2^{d-1}} \ll_{|\bftau|} (\log q)^s \left( q^{1-d}+ \min\{ q^{-1},\infty\}\right)^{\frac{s-\sig}{d-1}}.
\end{equation*}
This implies that 
\begin{equation*}
\sum_{\bfz\in \Z^s} \chi_l\left(\frac{\bfz}{q}\right)\ome_l \left(\frac{\bfz}{q}\right) e\left(\frac{r}{q}F(\bfz)\right) \ll_{|\bftau|}q^{s+\eps} \left(q^{-1}\right)^{2^{-d+1}\frac{s-\sig}{d-1}},
\end{equation*}
and hence the same bound holds for $S_{(I_1,I_2,\bftau)}(P;r,q)$.\par
We recall that $|I_1|+|I_2|+|I_3|=s$ for any tuple $(I_1,I_2,\bftau)\in \calI$. We bound a truncated version of the singular series $\grS_{(I_1,I_2,\bftau)}(P,n;Q)$ by
\begin{equation*}
\begin{split}
&\sum_{Q_1<q\leq Q_2} \sum_{\substack{r=1\\ (r,q)=1}}^q q^{-|I_3|+|\bftau|} \left| S_{(I_1,I_2,\bftau)}(P;r,q)e\left(-\frac{r}{q}n\right)\right| \\ \ll &\sum_{Q_1<q\leq Q_2} q^{-s+|I_1|+|I_2|+|\bftau|+1}q^{s+\eps} (q^{-1})^{2^{-d+1}\frac{s-\sig}{d-1}} \\ \ll & \sum_{Q_1<q\leq Q_2} q^{|I_1|+|I_2|+|\bftau|+1+\eps}q^{-2^{-d+1}\frac{s-\sig}{d-1}}.
\end{split}
\end{equation*}
Note that $|I_1|+|I_2|+|\bftau|<K$ for $(I_1,I_2,\bftau)\in \calI(K)$, and hence we can bound the last expression by
\begin{equation*}
\ll Q_1^{K+1-2^{-d+1}\frac{s-\sig}{d-1}+\eps},
\end{equation*}
for any $\eps >0$.
\end{proof}

\section{Singular integral}
In this section we study the singular integrals $\calJ_{(I_1,I_2,\bftau)}(n;Q)$ for $(I_1,I_2,\bftau)\in \calI$. Under suitable conditions on $F$ and the box $\calB$ we show that this is absolutely convergent and we give some rate of convergence as $Q\rightarrow \infty$. The analysis is inspired by the classical treatment as in \cite{Bir62}.\par
Assume as before that $F(\bfx)$ is a homogeneous form of degree $d$. Fix a partition of $\{1,\ldots, s\}$ into three index sets $I_i$, $1\leq i\leq 3$ and set $I_4=\emptyset$. Then $F(\sig_{\bfa,\bfb}(\bfx))$ is a polynomial in the variables $x_i$, $i\in I_3$ of degree $d'\leq d$. We assume that $d'=d$, and write as before $F^{[d]}(\sig_{\bfa,\bfb}(\bfx))$ for the homogeneous part of $F$ in $\bfx_{I_3}$ of degree $d$. The affine singular locus $\Sing (F^{[d]}(\sig_{\bfa,\bfb}(\bfx)))$ of $F^{[d]}(\sig_{\bfa,\bfb}(\bfx))$ in affine $|I_3|$-space is given by the system of equations
\begin{equation*}
\frac{\partial }{\partial x_i} F^{[d]}(\sig_{\bfa,\bfb}(\bfx)) = 0, \quad i\in I_3.
\end{equation*}
Let $\rho_{(I_1,I_2)}$ be the dimension of this affine variety, and note that it is independent of $\bfa$ and $\bfb$.

\begin{lemma}\label{integralconv}
Let $Q\geq 1$ and $(I_1,I_2,\bftau)\in \calI$. Assume that $F(\sig_{\bfa,\bfb}(\bfx))$ is of degree $d$ in $\bfx_{I_3}$, and 
\begin{equation}\label{integralconv1}
s-|I_1|-|I_2|-\rho_{(I_1,I_2)} >( |\bftau|+1) (d-1)2^{d-1}.
\end{equation}
Then $\calJ_{(I_1,I_2,\bftau)}(n;Q)$ is absolutely convergent and we have 
\begin{equation*}
|\calJ_{(I_1,I_2,\bftau)}(n;Q)-\calJ_{(I_1,I_2,\bftau)}(n)|\ll Q^{-2^{-d+1}\frac{s-|I_1|-|I_2|-\rho_{(I_1,I_2)}}{d-1}+1+|\bftau|+\eps}.
\end{equation*}
\end{lemma}

\begin{proof}
We recall that
\begin{equation}\label{eqnintconv1}
\calJ_{(I_1,I_2,\bftau)}(n;Q)= \int_{|\gam|\leq Q}J_{(I_1,I_2,\bftau)}(\gam)e(-\gam n)\d\gam,
\end{equation}
and
\begin{equation*}
J_{(I_1,I_2,\bftau)}(\gam)=\int_{\prod_{i\in I_3}[a_i,b_i]}f^{(\bftau)}(\gam, \sig_{\bfa,\bfb}(\bfx))\d \bfx_{I_3}.
\end{equation*}
It is clear that $|J_{(I_1,I_2,\bftau)}|\ll_{\bfa,\bfb} 1$ for all $\gam$ and hence we assume in the following that $|\gam|\geq 1$. Furthermore we first treat the case where $|\bftau|\geq 1$. We start the proof in rewriting the integral $J_{(I_1,I_2,\bftau)}(\gam)$ in the following way. By Lemma \ref{lemP1} one has
\begin{equation*}
f^{(\bftau)}(\gam,\sig_{\bfa,\bfb}(\bfx))= \sum_{l=1}^{|\bftau|} \gam^l h_l^{(\bftau)}(\sig_{\bfa,\bfb}(\bfx))e(\gam F(\sig_{\bfa,\bfb}(\bfx))),
\end{equation*}
with homogeneous polynomials $h_l^{(\bftau)}$ which are either identically zero of of degree $ld-|\bftau|$. Hence we have
\begin{equation*}
\begin{split}
J_{(I_1,I_2,\bftau)}(\gam)=& \int_{\prod_{i\in I_3}[a_i,b_i]}\sum_{l=1}^{|\bftau|}\gam^l h_l^{(\bftau)}(\sig_{\bfa,\bfb}(\bfx))e(\gam F(\sig_{\bfa,\bfb}(\bfx))) \d \bfx_{I_3}\\ 
=& \int_{\prod_{i\in I_3}[a_i,b_i]} \sum_{l=1}^{|\bftau|}(P^{-d}\gam)^l h_l^{(\bftau)}(\sig_{P\bfa,P\bfb}(P\bfx))P^{|\bftau|}\\ &e(P^{-d}\gam F(\sig_{P\bfa,P\bfb}(P\bfx))) \d \bfx_{I_3}.
\end{split}
\end{equation*}
A change of variables in the integral leads to
\begin{equation*}
\begin{split}
J_{(I_1,I_2,\bftau)}(\gam)=& P^{|\bftau|-|I_3|}\int_{P\prod_{i\in I_3}[a_i,b_i]} \sum_{l=1}^{|\bftau|}(P^{-d}\gam)^l h_l^{(\bftau)}(\sig_{P\bfa,P\bfb}(\bfx))\\ &e(P^{-d}\gam F(\sig_{P\bfa,P\bfb}(\bfx))) \d \bfx_{I_3} \\ 
=& P^{|\bftau|-|I_3|} \int_{P\prod_{i\in I_3}[a_i,b_i]}f^{(\bftau)}(P^{-d}\gam,\sig_{P\bfa,P\bfb}(\bfx))\d\bfx_{I_3}.
\end{split}
\end{equation*}
Next we approximate the last integral by a sum over integer tuples $\bfx\in \Z^{|I_3|}\cap P\prod_{i\in I_3}[a_i,b_i]$. For this we could us a form of Euler-MacLaurin summation formula, but for our purposes a much simpler argument is sufficient here. Note that if $|\bfx-\bfy|\leq 1$, then 
\begin{equation*}\begin{split}
\left|f^{(\bftau)}(P^{-d}\gam, \sig_{P\bfa,P\bfb}(\bfx))-f^{(\bftau)}(P^{-d}\gam, \sig_{P\bfa,P\bfb}\right.&\left.(\bfy))\right| \\&\ll \max_{j\in I_3} f^{(\bftau)+\bfe_j} (P^{-d}\gam, \sig_{P\bfa,P\bfb}(\bfxi)),
\end{split}
\end{equation*}
for $|\bfxi-\bfx|\leq 1$, where $\bfe_j$ is the $j$th unit vector. 
%consider the line between $\bfx$ and $\bfy$ and decompose the derivative into this direction into the different partial derivatives
If $|\bfx|\ll P$, then the decomposition for $f^{(\bftau)+\bfe_j} (P^{-d}\gam, \sig_{P\bfa,P\bfb}(\bfxi))$ as in Lemma \ref{lemP1} implies that 
\begin{equation*}
\begin{split}
\left|f^{(\bftau)}\right.-&\left.(P^{-d}\gam, \sig_{P\bfa,P\bfb}(\bfx)) f^{(\bftau)}(P^{-d}\gam, \sig_{P\bfa,P\bfb} (\bfy))\right| \\ &\ll  \sum_{1\leq l\leq |\bftau|+1} |P^{-d}\gam|^l P^{ld-|\bftau|-1} \ll \sum_{1\leq l\leq |\bftau|+1} |\gam|^l P^{-|\bftau|-1} \ll |\gam|^{|\bftau|+1} P^{-|\bftau|-1}.
\end{split}
\end{equation*}
Hence we can rewrite $J_{(I_1,I_2,\bftau)}(\gam)$ as 
\begin{equation*}
J_{(I_1,I_2,\bftau)}(\gam)= P^{|\bftau|-|I_3|} \sum_{\bfx_{I_3}\in P\prod_{i\in I_3}[a_i,b_i]}f^{(\bftau)}(P^{-d}\gam,\sig_{P\bfa,P\bfb}(\bfx)) + E_1 +E_2,
\end{equation*}
with an error term $E_1$ from the boundary of the box
\begin{equation*}
E_1\ll  P^{|\bftau|-|I_3|} P^{|I_3|-1} \sup_{\bfx\in P\calB} f^{(\bftau)} (P^{-d}\gam, \sig_{P\bfa,P\bfb}(\bfx)),
\end{equation*}
and an error term $E_2$ from approximating the sum by the integral in the interior of the box
\begin{equation*}
E_2\ll P^{|\bftau|-|I_3|} P^{|I_3|} |\gam|^{|\bftau|+1} P^{-|\bftau|-1}.
\end{equation*}
Using again the decomposition of $f^{(\bftau)}$ from Lemma \ref{lemP1} we can bound the first error term by
\begin{equation*}
E_1\ll P^{|\bftau|-1} \max_{1\leq l\leq |\bftau|}P^{-ld}|\gam|^l P^{ld-|\bftau|}\ll P^{-1}|\gam|^{|\bftau|}.
\end{equation*}
Recalling that we have assumed $|\gam|\geq 1$, we obtain
\begin{equation*}
J_{(I_1,I_2,\bftau)}(\gam)= P^{|\bftau|-|I_3|} \sum_{\bfx_{I_3}\in P\prod_{i\in I_3}[a_i,b_i]}f^{(\bftau)}(P^{-d}\gam,\sig_{P\bfa,P\bfb}(\bfx)) + O\left(P^{-1}|\gam|^{|\bftau|+1}\right).
\end{equation*}
Again, we use the decomposition of $f^{(\bftau)}(\gam,\bfx)$ as in Lemma \ref{lemP1} to decompose the main term as
\begin{equation}\label{eqnsingint3}
J_{(I_1,I_2,\bftau)}(\gam)= \sum_{l=1}^{|\bftau|} S_l + O\left(P^{-1}|\gam|^{|\bftau|+1}\right),
\end{equation}
with sums $S_l$ of the form
\begin{equation*}
S_l= P^{|\bftau|-|I_3|} \sum_{\bfx_{I_3}\in P\prod_{i\in I_3}[a_i,b_i]} (\gam P^{-d})^l h_l^{(\bftau)}(\sig_{P\bfa,P\bfb}(\bfx)) e\left(\gam P^{-d} F(\sig_{P\bfa,P\bfb}(\bfx))\right).
\end{equation*}
Using the homogeneity of the polynomials $h_l^{(\bftau)}(\bfx)$, we rewrite $S_l$ as
\begin{equation*}
S_l= P^{-|I_3|} \sum_{\bfx_{I_3}\in P\prod_{i\in I_3}[a_i,b_i]} \gam^l h_l^{(\bftau)}(\sig_{\bfa,\bfb}(\bfx/P)) e\left(\gam P^{-d} F(\sig_{P\bfa,P\bfb}(\bfx))\right).
\end{equation*}
%Assume that $F(\sig_{\bfa,\bfb}(\bfx))$ is of degree $d'\leq d$ in $\bfx$. Then $F^{[d']}(\sig_{\bfa,\bfb}(\bfx))$ is homogeneous of degree $d'$ in $\bfx$ and the coefficients are homogeneous polynomials in $\bfa,\bfb$ of degree $d-d'$. Hence 
%\begin{equation*}
%P^{-d+d'}F^{[d']}(\sig_{P\bfa,P\bfb}(\bfx))= F^{[d']}(\sig_{\bfa,\bfb}(\bfx)).
%\end{equation*}
We now apply Weyl's lemma in the form \ref{Weyl} to the exponential sum $S_l$. We will choose $P$ later and large enough such that $|\gam|<(1/2)P^{d}$. We apply Lemma \ref{Weyl} to the primitive tuple $1,\gam P^{-d}$. Note that our choice of $P$ implies that $|\gam P^{-d}|=\Vert \gam P^{-d}\Vert$. Furthermore, the quantity $\sig$ occurring in the exponent in Lemma \ref{Weyl} is in our case exactly the dimension of the singular locus of $F^{[d]}(\sig_{\bfa,\bfb}(\bfx))$, which we denoted by $\rho_{(I_1,I_2)}$. The polynomial $h_l^{(\bftau)}(\sig_{\bfa,\bfb}(\bfx))$ is a sum of monomials in $\bfx_{I_3}$ with coefficients depending polynomially on $\bfa$ and $\bfb$. We consider each monomial separately. The product of each of these with the indicator function $\id_{\prod_{i\in I_3}[a_i,b_i]} (\bfx)$ can be decomposed into $\sum_m \chi_m\ome_m$ as in the proof of Lemma \ref{singseries}, with some indicator functions $\chi_m$ and $\ome_m \in \calS (c,C,s)$ for $c$ and $C$ some positive constants only depending on $\bftau$, $|\bfa|$, $|\bfb|$ and $F$. Hence we obtain
\begin{equation*}\begin{split}
|\gam^{-l}S_l|^{2^{d-1}}\ll_{c,C}& (\log P)^{|I_3|} \\&\left( P^{1-d}+|P^{-d}\gam|+P^{-d} + \min\left\{ 1,\frac{1}{|\gam P^{-d}|P^{d}}\right\} \right)^{\frac{|I_3|-\rho_{(I_1,I_2)}}{d-1}}.
\end{split}
\end{equation*}
%note that in Lemma \ref{Weyl} only the height of the leading form needs to be uniformly bounded, the rest disappears during Weyl differencing and may have arbitrarily large coefficients
Now we choose $P$ sufficiently large depending on $|\gam|$ such that
\begin{equation*}
|\gam^{-l}S_l|^{2^{d-1}}\ll |\gam|^\eps |\gam|^{-\frac{|I_3|-\rho_{(I_1,I_2)}}{d-1}}.
\end{equation*}
Again choosing $P$ sufficiently large (a suitable power of $|\gam|$) we see from equation (\ref{eqnsingint3}) that
\begin{equation*}
J_{(I_1,I_2,\bftau)}(\gam)\ll \sum_{l=1}^{|\bftau|}|\gam|^{\eps-2^{-d+1}\frac{|I_3|-\rho_{(I_1,I_2)}}{d-1}+l},
\end{equation*}
and hence
\begin{equation*}
J_{(I_1,I_2,\bftau)}(\gam)\ll |\gam|^{\eps-2^{-d+1}\frac{|I_3|-\rho_{(I_1,I_2)}}{d-1}+|\bftau|}.
\end{equation*}
%implicit contant depending on \bftau,\bfa,\bfb,C
The assumption in (\ref{integralconv1}) now shows that the integral defining $\calJ_{(I_1,I_2,\bftau)}$ in equation (\ref{eqnintconv1}) is absolutely convergent. Similarly the second part of the lemma immediately follows. For $\bftau=0$ the same arguments (in a simplified form) reduce to the classical way of bounding the singular integral and hence the proof of the lemma follows also in this case.
\end{proof}

\section{Proof of the main theorems}
In this section we collect together the information about the major and minor arcs and give a proof of Theorem \ref{thm1}, followed with a deduction of Theorem \ref{thm2}. Before, we shortly give an easy upper bound for the size of the singular loci $\rho_{(I_1,I_2)}$ in terms of the singular locus $\sig$.

\begin{lemma}\label{dimsingloci}
%Assume that $s-\sig > 2(|I_1|+|I_2|)$.
For any $I_1$ and $I_2$ one has
\begin{equation*}
\rho_{(I_1,I_2)}\leq \sig +|I_1| +|I_2|.
\end{equation*}
If $s-\sig > 2(|I_1|+|I_2|)$, then the homogeneous part $F^{[d]}(\sig_{\bfa,\bfb}(\bfx))$ is not identically zero.
\end{lemma}

\begin{proof}
Recall that $F(\bfx)$ is a homogeneous form of degree $d$, and note that the homogeneous part $F^{[d]}(\sig_{\bfa,\bfb}(\bfx))$ if independent of $\bfa$ and $\bfb$. We write 
\begin{equation*}
F(\bfx)= F^{[d]}(\sig_{\bfa,\bfb}(\bfx))+\sum_{i\in I_1\cup I_2}x_iH_i(\bfx),
\end{equation*}
with $H_i(\bfx)$ homogeneous polynomials of degree $d-1$ in all of the variables $x_i$, $1\leq i\leq s$. Let $Y$ be the affine variety given by the system of equations
\begin{eqnarray}\label{eqn8.3}
\frac{\partial}{\partial x_i}F^{[d]}(\sig_{\bfa,\bfb}(\bfx))&=0,\quad i\in I_3\\
x_i=H_i(\bfx)&=0,\quad i\in I_1\cup I_2.
\end{eqnarray}
Then we have $Y\subset \Sing (F(\bfx))$, and hence $\dim Y\leq \sig$. On the other hand we may consider the affine variety $Y'\subset \A^s$ given by the system of equations (\ref{eqn8.3}) only. By definition of $\rho_{(I_1,I_2)}$ we have
\begin{equation}\label{eqn8.3b}
\dim Y'= \rho_{(I_1,I_2)}+|I_1|+|I_2|.
\end{equation}
Since all the polynomials defining $Y$ and $Y'$ are homogeneous, we have
\begin{equation*}
\dim Y\geq \dim Y'-2(|I_1|+|I_2|).
\end{equation*}
Together with equation (\ref{eqn8.3b}) this implies
\begin{equation*}
\rho_{(I_1,I_2)}\leq \sig+|I_1|+|I_2|.
\end{equation*}
In particular we have $F^{[d]}(\sig_{\bfa,\bfb}(\bfx))\neq 0$ as soon as 
\begin{equation*}
\sig+|I_1|+|I_2|<s-(|I_1|+|I_2|).
\end{equation*}
\end{proof}

We now come to the proof of our main theorem \ref{thm1}. 

\begin{proof}[Proof of Theorem \ref{thm1}]
Let $0<\eta <(1/3)d$ to be chosen later, and $K\geq 1$ as in Theorem \ref{thm1}. 
%The case $K=0$ reduces to the main theorem in Birch's work \cite{Bir62}, as for example in his Theorem 1, and hence there is nothing to show. Thus we assume in the following that $K\geq 1$. Next
We decompose the counting function $R(P,n)$ as
\begin{equation}\label{eqn8.1}
R(P,n)=R(P,n;\grM'(\eta)) + O\left(\int_{\grm(\eta)}|S(\alp)|\d\alp\right).
\end{equation}
By Lemma \ref{minarcs} the contribution of the minor arcs is bounded by
\begin{equation}\label{eqn8.2}
\int_{\grm(\eta)}|S(\alp)|\d\alp\ll_\eps P^{s-d+\eps - \eta \left(2^{-d+1}\frac{s-\sig}{d-1}-2\right)},
\end{equation}
for any $\eps >0$. The major arc contribution is by Lemma \ref{MA2} given by
\begin{equation}\label{eqn8.4}
\begin{split}
R(P,n;\grM')&= \sum_{(I_1,I_2,\bftau)\in \calI(K)} \grS_{(I_1,I_2,\bftau)}(P,n;P^\eta)\calJ_{(I_1,I_2,\bftau)} (P^{-d n};P^\eta) P^{|I_3|-|\bftau|-d} \\ &+O\left( P^{s-K-d+(2K+3)\eta}\right).
\end{split}
\end{equation}
We next complete the singular series and singular integral in each term appearing in the sum over $(I_1,I_2,\bftau)\in \calI(K)$. Note that the number of terms in the summation is bounded by $|\calI(K)|\ll_K 1$. By Lemma \ref{singseries} we have
\begin{equation*}
\grS_{(I_1,I_2,\bftau)}(P,n;P^\eta)-\grS_{(I_1,I_2,\bftau)}(P,n)\ll P^{\eta\left(K+1-2^{-d+1}\frac{s-\sig}{d-1}\right)+\eps},
\end{equation*}
for any $\eps >0$, as soon as
\begin{equation*}
2^{-d+1}\frac{s-\sig}{d-1}>K+1.
\end{equation*}
In particular the proof of Lemma \ref{singseries} shows that both $\grS_{(I_1,I_2,\bftau)}(P,n;P^\eta)$ and $\grS_{(I_1,I_2,\bftau)}(P,n)$ are bounded by $\ll 1$. For the singular integrals we use Lemma \ref{integralconv} and observe that 
\begin{equation*}
|\calJ_{(I_1,I_2,\bftau)}(n;P^\eta)-\calJ_{(I_1,I_2,\bftau)}(n)|\ll P^{-\eta\left(2^{-d+1}\frac{s-|I_1|-|I_2|-\rho_{(I_1,I_2)}}{d-1}-1-|\bftau|\right)+\eps},
\end{equation*}
as soon as equation (\ref{integralconv1}) holds.
%\begin{equation}\label{integralconv1}
%s-|I_1|-|I_2|-\sig_{(I_1,I_2)} >( |\bftau|+1) (d-1)2^{d-1}.
%\end{equation}
We replace in equation (\ref{eqn8.4}) the truncated singular integral $\calJ_{(I_1,I_2,\bftau)}(n;P^\eta)$ by $\calJ_{(I_1,I_2,\bftau)}(n)$ and $\grS_{(I_1,I_2,\bftau)}(P,n;P^\eta)$ by $\grS_{(I_1,I_2,\bftau)}(P,n)$. We obtain 
\begin{equation}\label{eqn8.5}
\begin{split}
R(P,n;\grM')&= \sum_{(I_1,I_2,\bftau)\in \calI(K)} \grS_{(I_1,I_2,\bftau)}(P,n)\calJ_{(I_1,I_2,\bftau)} (P^{-d}n) P^{|I_3|-|\bftau|-d} \\ &+O\left( P^{s-K-d+(2K+3)\eta}\right)+E_1+E_2,
\end{split}
\end{equation}
with error terms of the form
\begin{equation}\label{eqn8.6}
\begin{split}
E_1&\ll \sum_{(I_1,I_2,\bftau)\in \calI(K)}P^{\eta \left(K+1-2^{-d+1}\frac{s-\sig}{d-1}\right)}P^{|I_3|-|\bftau|-d+\eps} \\ 
&\ll \sum_{(I_1,I_2,\bftau)\in \calI(K)}P^{|I_3|-|\bftau|-d-\eta\left(2^{-d+1}\frac{s-\sig}{d-1}-K-1\right)+\eps} \\
&\ll P^{s-d-\eta\left(2^{-d+1}\frac{s-\sig}{d-1}-K-1\right)+\eps},
\end{split}
\end{equation}
and
\begin{equation}\label{eqn8.7}
E_2\ll \sum_{(I_1,I_2,\bftau)\in \calI(K)} P^{|I_3|-|\bftau|-d+\eps} P^{-\eta\left(2^{-d+1}\frac{|I_3|-\rho_{(I_1,I_2)}}{d-1}-1-|\bftau|\right)}.
%\end{split}
\end{equation}
We now compare the different error terms. First we note that the bound for $E_1$ in (\ref{eqn8.6}) is weaker than the bound for the minor arc contribution in (\ref{eqn8.2}). Furthermore, we can estimate an individual term in the bound for $E_2$ in (\ref{eqn8.7}) by
\begin{equation*}
\begin{split}
P^{|I_3|-|\bftau|-d+\eps} &P^{-\eta\left(2^{-d+1}\frac{|I_3|-\rho_{(I_1,I_2)}}{d-1}-1-|\bftau|\right)}\\&\ll  P^{s-d+\eps - |I_1|-|I_2|-(1-\eta)|\bftau|-\eta \left(2^{-d+1}\frac{|I_3|-\rho_{(I_1,I_2)}}{d-1}-1\right)}.
\end{split}
\end{equation*}
Assume that $\eta <1$ and $s-\sig > 2(|I_1|+|I_2|)$. Together with Lemma \ref{dimsingloci} we obtain
\begin{equation*}
\begin{split}
|I_1|&+|I_2| +\eta\left(2^{-d+1}\frac{|I_3|-\rho_{(I_1,I_2)}}{d-1}-1\right) \\&\geq |I_1|+|I_2| +\eta\left(2^{-d+1}\frac{s-\sig-2(|I_1|+|I_2|)}{d-1}-1\right)
\geq \eta \left(2^{-d+1}\frac{s-\sig}{d-1}-1\right).
\end{split}
\end{equation*}
%since 1-2^{-d+1}\frac{2}{d-1}\eta >0
Hence we see that 
\begin{equation*}
E_2\ll P^{s-d+\eps - \eta \left(2^{-d+1}\frac{s-\sig}{d-1}-1\right)},
\end{equation*}
which is a better bound than what we obtained for $E_1$ in equation (\ref{eqn8.6}). From equation (\ref{eqn8.1}) and equation (\ref{eqn8.5}) and the bounds in (\ref{eqn8.2}) and (\ref{eqn8.6}) we conclude that
\begin{equation}\label{eqn8.10}
\begin{split}
R(P,n)&= \sum_{(I_1,I_2,\bftau)\in \calI(K)} \grS_{(I_1,I_2,\bftau)}(P,n)\calJ_{(I_1,I_2,\bftau)} (P^{-d}n) P^{|I_3|-|\bftau|-d} \\ &+O\left( P^{s-K-d+(2K+3)\eta}\right)+O\left( P^{s-d+\eps-\eta\left(2^{-d+1}\frac{s-\sig}{d-1}-K-1\right)}\right).
\end{split}
\end{equation}
We choose $\eta$ such that
\begin{equation*}
-K+(2K+3)\eta = -\eta \left(2^{-d+1}\frac{s-\sig}{d-1}-K-1\right),
\end{equation*}
which is equivalent to 
\begin{equation*}
K=\eta \left(2^{-d+1}\frac{s-\sig}{d-1}+K+2\right).
\end{equation*}
Note that the assumption $2^{-d+1}\frac{s-\sig}{d-1}>K+1$ ensures that 
\begin{equation*}
\eta < \frac{K}{2K+3}<\frac{1}{2}\leq \frac{1}{3}d,
\end{equation*}
and hence our assumption above on $\eta <1$ is justified, as well as the major arcs are disjoint as required in Lemma \ref{MA2}. Furthermore, the assumption 
\begin{equation*}
2^{-d+1}\frac{s-\sig}{d-1}>2K^2+2K-2
\end{equation*}
in the theorem ensures that
\begin{equation*}
\eta < K ( 2K^2+3K)^{-1}= (2K+3)^{-1}.
\end{equation*}
Hence we obtain 
\begin{equation}\label{eqn8.10}
\begin{split}
R(P,n)&= \sum_{(I_1,I_2,\bftau)\in \calI(K)} \grS_{(I_1,I_2,\bftau)}(P,n)\calJ_{(I_1,I_2,\bftau)} (P^{-d}n) P^{|I_3|-|\bftau|-d} \\ &+O\left( P^{s-d-(K-1)-\del}\right),
\end{split}
\end{equation}
for this choice of $\eta$ and some $\del >0$.
\end{proof}

Finally we deduce Theorem \ref{thm2} from Theorem \ref{thm1}

\begin{proof}[Proof of Theorem \ref{thm2}]
Recall that
\begin{equation}\label{eqnthm2p}
\begin{split}
N_X(P)&=\frac{1}{2}\sum_{e=1}^\infty \mu(e) \left(R_\calB(e^{-1}P,0)-1\right)\\
&=\frac{1}{2}\sum_{e=1}^{[P]} \mu(e) \left(R_\calB(e^{-1}P,0)-1\right).
\end{split}
\end{equation}
By Theorem \ref{thm1} we have for any $e\leq P$
\begin{equation*}
\begin{split}
R_\calB(e^{-1}P,0)-1&= \sum_{(I_1,I_2,\bftau)\in \calI(K)}\grS_{(I_1,I_2,\bftau)}(e^{-1}P,0)\calJ_{(I_1,I_2,\bftau)}(0)(e^{-1}P)^{s-|I_1|-|I_2|-|\bftau|-d}\\ &+O\left((e^{-1}P)^{s-d-(K-1)-\del}\right).
\end{split}
\end{equation*}
By Lemma \ref{singseries} we observe that
\begin{equation*}
\begin{split}
\widetilde{\grS}_{(I_1,I_2,\bftau)}(P)= &\frac{1}{2}\sum_{e=1}^{[P]}\mu(e) e^{-(s-|I_1|-|I_2|-|\bftau|-d)}\grS_{(I_1,I_2,\bftau)}(e^{-1}P,0) \\ &+O(P^{-(s-|I_1|-|I_2|-|\bftau|-d)+1}).
\end{split}
\end{equation*}
Putting the higher order asymptotic expansions for $R_\calB(e^{-1}P,0)-1$ into equation (\ref{eqnthm2p}) finally leads to
\begin{equation*}
\begin{split}
N_X(P)=& \sum_{(I_1,I_2,\bftau)\in \calI(K)} \widetilde{\grS}_{(I_1,I_2,\bftau)}(P)\calJ_{(I_1,I_2,\bftau)} (0) P^{s-|I_1|-|I_2|-|\bftau|-d} \\ &+O\left( P^{s-d-(K-1)-\del}\right).
\end{split}
\end{equation*}
\end{proof}

\section{Singular integral II}\label{SingIntII}
In this section we come back to the study of the singular integrals $\calJ_{(I_1,I_2,\bftau)}(n)$. We shortly recall the definitions
\begin{equation*}
J_{(I_1,I_2,\bftau)}(\gam)=\int_{\prod_{i\in I_3}[a_i,b_i]}f^{(\bftau)}(\gam,\sig_{\bfa,\bfb}(\bfx))\d \bfx_{I_3},
\end{equation*}
and
\begin{equation*}
\calJ_{(I_1,I_2,\bftau)}(n)=\int_{-\infty}^\infty J_{(I_1,I_2,\bftau)}(\gam)e(-\gam n)\d\gam,
\end{equation*}
which is absolutely convergent under condition (\ref{integralconv1}) by Lemma \ref{integralconv}. Recall that we have set
\begin{equation*}
f^{(\bftau)}(\gam,\sig_{\bfa,\bfb}(\bfx))= \partial_\bfx^{\bftau}e(\gam F(\bfx))|_{\bfx=\sig_{\bfa,\bfb}(\bfx)}.
\end{equation*}
Let $\bfx_{I_1}$ and $\bfx_{I_2}$ be vectors defined in an analogous way as $\bfx_{I_3}$, i.e. $\bfx_{I_1}=(x_i)_{i\in I_1}$ and $\bfx_{I_2}= (x_i)_{i\in I_2}$. As a generalisation of $J_{(I_1,I_2,\bftau)}(\gam)$, it is convenient to define 
\begin{equation*}
J_{(I_1,I_2,\bftau)}(\gam;\bfx_{I_1},\bfx_{I_2})=\int_{\prod_{i\in I_3}[a_i,b_i]}f^{(\bftau)}(\gam,\bfx)\d \bfx_{I_3},
\end{equation*}
and
\begin{equation*}
\calJ_{(I_1,I_2,\bftau)}(n;\bfx_{I_1},\bfx_{I_2})=\int_{-\infty}^\infty J_{(I_1,I_2,\bftau)}(\gam;\bfx_{I_1},\bfx_{I_2})e(-\gam n)\d\gam.
\end{equation*}
This integral is absolutely convergent for $\bfx_{I_1}\in \prod_{i\in I_1} [a_i,b_i]$ and $\bfx_{I_2}\in \prod_{i\in I_2}[a_i,b_i]$, as soon as (\ref{integralconv1}) holds. Note that we have
\begin{equation*}
J_{(I_1,I_2,\bftau)}(\gam)=J_{(I_1,I_2,\bftau)}(\gam;\bfb_{I_1},\bfa_{I_2}),
\end{equation*}
and
\begin{equation*}
\calJ_{(I_1,I_2,\bftau)}(n)=\calJ_{(I_1,I_2,\bftau)}(n;\bfb_{I_1},\bfa_{I_2}).
\end{equation*}
Furthermore, we write $\calJ_{(I_2,I_2)}(n;\bfx_{I_1},\bfx_{I_2})$ for $\calJ_{(I_1,I_2,\mathbf{0})}(n;\bfx_{I_1},\bfx_{I_2})$.\par
In order to give a different description of $\calJ_{(I_1,I_2,\bftau)}(n;\bfx_{I_1},\bfx_{I_2})$, we proceed in a similar way as Schmidt in his work \cite{Schmidt80}. However, in contrast to his treatment we need to introduce some different kernel with sufficient decay. We choose to use the smooth and compactly supported weight function
\begin{align*}
\ome(x)=\left\{\begin{array}{ccc}c_0 e^{-(1-x^2)^{-1}} & \mbox{ for } & |x|<1 \\ 0 & \mbox{ for } & |x|\geq 1,\end{array}\right.
\end{align*}
where $c_0$ is a normalisation constant such that $\int_\R \ome(x)\d x=1$. Let 
\begin{equation*}
\hat{\ome}(y)=\int_\R \ome(\gam)e(\gam y)\d\gam,
\end{equation*}
be the Fourier transform of $\ome$. Then we have $\hat{\ome}(0)=1$ and $\omehat(x)=\omehat(0)+O(|x|)$. 
%Furthermore, we set $\omehat_T(x)=\omehat(T^{-1}x)$ for any real $T\geq 1$. 
We now define the modified singular integrals as 
\begin{equation*}
\calJ^{T}_{(I_1,I_2,\bftau)}(n;\bfx_{I_1},\bfx_{I_2})=\int_{-\infty}^\infty \omehat(T^{-1}\gam)J_{(I_1,I_2,\bftau)}(\gam;\bfx_{I_1},\bfx_{I_2})e(-\gam n)\d\gam.
\end{equation*}
Next we observe that $\calJ^{T}_{(I_1,I_2,\bftau)}(n;\bfx_{I_1},\bfx_{I_2})$ is a good model for the original singular integral as $T$ gets large.

\begin{lemma}\label{SingInta}
Assume that (\ref{integralconv1}) holds. Then one has
\begin{equation*}
\calJ^{T}_{(I_1,I_2,\bftau)}(n;\bfx_{I_1},\bfx_{I_2})\rightarrow \calJ_{(I_1,I_2,\bftau)}(n;\bfx_{I_1},\bfx_{I_2}),\quad \mbox{ for } T\rightarrow \infty,
\end{equation*}
and the convergence is uniform in $\bfx_{I_1}\in \prod_{i\in I_1}[a_i,b_i]$ and $\bfx_{I_2}\in \prod_{i\in I_2}[a_i,b_i]$.
\end{lemma}

%note that by Lemma \ref{dimsingloci} one has $\sig_{(I_1,I_2)}\leq \sig +|I_1|+|I_2|$ and hence $s-|I_1|-|I_2|-\sig_{(I_1,I_2)}\geq s-\sig -2(|I_1|+|I_2|) and one could replace \ref{integralconv1} by 
%s-\sig-2(|I_1|+|I_2|)> (d-1)2^{d-1}(|\bftau|+1).

\begin{proof}
We bound the difference in the lemma by
\begin{equation*}
\calJ_{(I_1,I_2,\bftau)}(n;\bfx_{I_1},\bfx_{I_2})-\calJ^{T}_{(I_1,I_2,\bftau)}(n;\bfx_{I_1},\bfx_{I_2})\ll A+B,
\end{equation*}
with
\begin{equation*}
A= \int_{|\gam|\leq T}|1-\omehat(T^{-1}\gam)||J_{(I_1,I_2,\bftau)}(\gam;\bfx_{I_1},\bfx_{I_2})|\d\gam,
\end{equation*}
and
\begin{equation*}
\begin{split}
B=&\int_{|\gam|>T}|1-\omehat(T^{-1}\gam)| |J_{(I_1,I_2,\bftau)}(\gam;\bfx_{I_1},\bfx_{I_2})|\d\gam\\ \ll  &\int_{|\gam|>T}  |J_{(I_1,I_2,\bftau)}(\gam;\bfx_{I_1},\bfx_{I_2})|\d\gam.
\end{split}
\end{equation*}
We recall by the proof of Lemma \ref{integralconv} that for any $\eps >0$ we have
\begin{equation*}
J_{(I_1,I_2,\bftau)}(\gam;\bfx_{I_1},\bfx_{I_2})\ll_\eps |\gam|^{\eps-2^{-d+1}\frac{|I_3|-\sig_{(I_1,I_2)}}{d-1}+|\bftau|},
\end{equation*}
uniformly in $\bfx_{I_1}\in \prod_{i\in I_1}[a_i,b_i]$ and $\bfx_{I_2}\in \prod_{i\in I_2}[a_i,b_i]$. Hence we can bound the first term by
\begin{equation*}
\begin{split}
A\ll &\int_{|\gam|\leq T} |\gam T^{-1}|\min\left(1,|\gam|^{\eps-2^{-d+1}\frac{|I_3|-\rho_{(I_1,I_2)}}{d-1}+|\bftau|}\right)\d\gam\\ \ll &\int_{|\gam|<1}T^{-1}\d\gam+ \int_{1\leq |\gam|\leq T}  |\gam T^{-1}||\gam|^{\eps-2^{-d+1}\frac{|I_3|-\rho_{(I_1,I_2)}}{d-1}+|\bftau|}\d\gam\ll T^{-\del},
\end{split}
\end{equation*}
for some $\del >0$ as soon as 
\begin{equation*}
|I_3|-\rho_{(I_1,I_2)}> (|\bftau|+1)(d-1)2^{d-1}.
\end{equation*}
The same argument shows that $B\ll T^{-\del}$ under condition (\ref{integralconv1}) and hence the lemma follows.
\end{proof}

 Next we aim to find some interpretation for the integral $\calJ^T_{(I_1,I_2)}(n;\bfx_{I_1},\bfx_{I_2})$. By Fubini's theorem we have
\begin{equation*}
\begin{split}
\calJ^T_{(I_1,I_2)}(n;\bfx_{I_1},\bfx_{I_2})&= \int_{\prod_{i\in I_3}[a_i,b_i]}\int_\R \omehat(T^{-1}\gam)e(\gam(F(\bfx)-n))\d\gam\d\bfx_{I_3} \\ &= T \int_{\prod_{i\in I_3}[a_i,b_i]}\int_\R \omehat(\gam)e(T\gam(F(\bfx)-n))\d\gam\d\bfx_{I_3}\\
&= T\int_{\prod_{i\in I_3}[a_i,b_i]}\ome(T(F(\bfx)-n))\d\bfx_{I_3}.
\end{split}
\end{equation*}
If we set $\ome_T(y)=T\ome(T y)$, then we can rewrite the last equation as
\begin{equation}\label{SingIntVol}
\calJ^T_{(I_1,I_2)}(n;\bfx_{I_1},\bfx_{I_2})= \int_{\prod_{i\in I_3}[a_i,b_i]}\ome_T(F(\bfx)-n)\d\bfx_{I_3}.
\end{equation}
For $T\rightarrow \infty$, the integral $\calJ^T_{(I_1,I_2)}(n;\bfx_{I_1},\bfx_{I_2})$ hence converges to the volume of the bounded piece of the hypersurface $F(\bfx)=n$ inside the box $\prod_{i\in I_3}[a_i,b_i]$, where $\bfx_{I_1}$ and $\bfx_{I_2}$ are considered fixed. By Lemma \ref{SingInta} this limit equals the singular integral $\calJ_{(I_1,I_2)}(n;\bfx_{I_1},\bfx_{I_2})$.\par

The following lemma relates the singular integrals $\calJ_{(I_1,I_2,\bftau)}(n)$ for non-zero $\bftau$ to the function $\calJ_{(I_1,I_2)}(n;\bfx_{I_1},\bfx_{I_2})$ and hence gives a natural interpretation for these kinds of singular integrals. 

\begin{lemma}\label{SingIntb}
Assume that (\ref{integralconv1}) holds. Then one has
\begin{equation*}
\calJ_{(I_1,I_2,\bftau)}(n)= \partial_\bfx^{\bftau}\calJ_{(I_1,I_2)}(n;\bfx_{I_1},\bfx_{I_2})|_{\bfx=\sig_{\bfa,\bfb}(\bfx)}.
\end{equation*}
In particular, the singular integral $\calJ_{(I_1,I_2,\bftau)}(n)$ is the partial derivative $\partial^{\bftau}_\bfx$ of the function in $\bfx_{I_1}$ and $\bfx_{I_2}$ describing the volume of the bounded piece of the hypersurface $F(\bfx)=n$ inside the box $\prod_{i\in I_3}[a_i,b_i]$, at the point $\bfx_{I_1}=\bfb$ and $\bfx_{I_2}=\bfa$.
\end{lemma}

\begin{proof}
Recall that
\begin{equation*}
\calJ_{(I_1,I_2)}(n;\bfx_{I_1},\bfx_{I_2})=\int_\R e(-\gam n) \int_{\prod_{i\in I_3}[a_i,b_i]}e(\gam F(\bfx))\d\bfx_{I_3}\d\gam.
\end{equation*}
%By Lemma \ref{SingInta} we have 
%\begin{equation*}
%\calJ^T_{(I_1,I_2)}(n;\bfx_{I_1},\bfx_{I_2})\rightarrow \calJ_{(I_1,I_2)}(n;\bfx_{I_1},\bfx_{I_2}),
%\end{equation*}
%as $T\rightarrow \infty$. 
The proof of Lemma \ref{integralconv} shows that $\calJ_{(I_1,I_2)}(n;\bfx_{I_1},\bfx_{I_2})$ as an integral with respect to the integration variable $\gam$, is uniformly convergent with respect to $\bfx_{I_1}\in \prod_{i\in I_1}[a_i-1,b_i+1]$ and $\bfx_{I_2}\in \prod_{i\in I_2}[a_i-1,b_i+1]$. The same holds for the integral
\begin{equation*}
\begin{split}
\calJ_{(I_1,I_2,\bftau)}(n;\bfx_{I_1},\bfx_{I_2})&=\int_\R e(-\gam n) \partial_\bfx^{\bftau}\int_{\prod_{i\in I_3}[a_i,b_i]}e(\gam F(\bfx))\d\bfx_{I_3}\d\gam \\ &=\int_\R e(-\gam n) \int_{\prod_{i\in I_3}[a_i,b_i]}\partial_\bfx^{\bftau}e(\gam F(\bfx))\d\bfx_{I_3}\d\gam.
\end{split}
\end{equation*}
%converges to $\calJ_{(I_1,I_2,\bftau)}(n,\bfx_{I_1},\bfx_{I_2})$ for $T\rightarrow \infty$, uniformly in $\bfx_{I_1}\in \prod_{i\in I_1}[a_i-1,b_i+1]$ and $\bfx_{I_2}\in \prod_{i\in I_2}[a_i-1,b_i+1]$. 
Hence we see that by Leibniz' rule we have
\begin{equation*}
\calJ_{(I_1,I_2,\bftau)}(n)= \partial_\bfx^{\bftau}\calJ_{(I_1,I_2)}(n;\bfx_{I_1},\bfx_{I_2})|_{\bfx=\sig_{\bfa,\bfb}(\bfx)},
\end{equation*}
which proves the lemma.
%it is clear that by Leibniz rule 
%$\partial_\bfx^\bftau \int_{\prod_{i\in I_3}[a_i,b_i]}e(\gam F(\bfx))\d\bfx_{I_3}= \int_{\prod_{i\in I_3}[a_i,b_i]}\partial_\bfx^\bftau e(\gam F(\bfx))\d\bfx_{I_3}$ 
%hence we may use: let $f$ and $\partial f/\partial y$ be continuous in [a,infty)\times [c,d] and suppose that \int_a^\infty f dx and \int_a^\infty \partial f/\partial ydx are uniformly conergent. then phi(y)=\int_a^\infty f(x,y)\d x is differentiabe and \frac{\parital \phi}{\parital y}= \int_a^\infty \frac{\partial f}{\partial y}(x,y)d x. 
\end{proof}

\section{Singular Series II}\label{SingSeriesII}
In this section we give some interpretation of the singular series $\grS_{(I_1,I_2,\bftau)}(P,n)$. If $I_1=I_2=\emptyset$ (and hence $\bftau=0$), then this series reduces to the classical singular series. The function $S_{(\emptyset,\emptyset,\mathbf{0})}(P;r,q)$ is multiplicative in $q$ in a sense that 
\begin{equation*}
S_{(\emptyset,\emptyset,\mathbf{0})}(P;r,q) S_{(\emptyset,\emptyset,\mathbf{0})}(P;r',q') =S_{(\emptyset,\emptyset,\mathbf{0})}(P;rq'+r'q,qq')
\end{equation*}
for coprime moduli $(q,q')=1$. This leads to an expression of $\grS_{(\emptyset,\emptyset,\mathbf{0})}(P,n)$ as a product of local densities. However, if not both of $I_1$ and $I_2$ are empty we do not expect the same multiplicative behaviour of $S_{(I_1,I_2,\bftau)}(P;r,q)$ because of the presence of the Bernoulli polynomials $\bet_{\tau_i+1}\left(\frac{Pb_i-z_i}{q}\right)$. Hence we cannot expect to factorize $\grS_{(I_1,I_2,\bftau)}(P,n)$ in the traditional way. In order to get some interpretation for these terms, we take the following approach. We truncate the series $\grS_{(I_1,I_2,\bftau)}(P,n)$ at some height $q\leq Q$ and interpret the truncated singular series up to a small error as a weighted number of local solutions modulo $Q!$. For this we need to lift the denominators in the exponential sums $S_{(I_1,I_2,\bftau)}(P;r,q)$ all to the same denominator. In the case of the classical singular series one clearly has
\begin{equation*}
S_{(\emptyset,\emptyset,\mathbf{0})}(P;dr,dq)= d^s S_{(\emptyset,\emptyset,\mathbf{0})}(P;r,q).
\end{equation*}
In the case of generalized exponential sums $S_{(I_1,I_2,\bftau)}(P;r,q)$, which may contain products of Bernoulli polynomials, this is less obvious. In this case the analogous observation is a consequence of the multiplication theorem for Bernoulli numbers. This states that for any $d\geq 1$ one has
\begin{equation}\label{Bernoulli}
B_n(dx)=d^{n-1}\sum_{k=0}^{d-1}B_n\left(x+\frac{k}{d}\right),
\end{equation}
see for example \cite{How95} for a reference.

\begin{lemma}\label{lll}
Assume that $d\geq 1$. Then one has
\begin{equation*}
S_{(I_1,I_2,\bftau)}(P;dr,dq)= d^{|I_3|-|\bftau|}S_{(I_1,I_2,\bftau)}(P;r,q).
\end{equation*}
\end{lemma}

\begin{proof}
For simplicity of notation we write 
\begin{equation*}
\varrho_{(I_1,I_2,\bftau)}:= (-1)^{|I_2|}\prod_{i\in I_1\cup I_2}\frac{(-1)^{\tau_i+1}}{(\tau_i+1)!}.
\end{equation*}
Then we can write the exponential sum of interest as
\begin{equation*}
\begin{split}
S_{(I_1,I_2,\bftau)}(P;dr,dq)= &\varrho_{(I_1,I_2,\bftau)}\sum_{\bfz' \mmod q} e\left(\frac{r}{q}F(\bfz')\right)\prod_{i\in I_1} \bet_{\tau_i+1}\left(\frac{Pb_i-z_i'}{dq}\right)\\ &\prod_{i\in I_2}\bet_{\tau_i+1}\left(\frac{Pa_i-z_i'}{dq}\right).
\end{split}
\end{equation*}
Next we rewrite the variables $z_i'$ in the summation as $z_i'=z_i+qh_i$ with $0\leq h_i<d$ and $z_i$ running through an interval of length $q$, such that the following holds. If $i\in I_1$ one has
\begin{equation*}
-1\leq \frac{Pb_i-z_i}{dq}-\frac{h_i}{d}<0
\end{equation*}
for all choices of $0\leq h_i<d$. Similarly for $i\in I_2$ or $i$ not contained in $I_1\cup I_2$. If for example $i\in I_1$, then one has 
\begin{equation*}
-1\leq \frac{Pb_i-z_i}{q}<0.
\end{equation*}
%since -d+h_i \leq \frac{Pb_i-z_i}{q}<h_i for all 0\leq h_i <d
For $i\in I_1$ we need to compute the sum
\begin{equation*}
\begin{split}
\sum_{0\leq h_i<d}\bet_{\tau_i+1}\left(\frac{Pb_i-z_i-qh_i}{dq}\right) &= \sum_{0\leq h_i<d}B_{\tau_i+1}\left(1+\frac{Pb_i-z_i}{dq}-\frac{h_i}{d}\right)\\ 
&= \sum_{0\leq h_i<d}B_{\tau_i+1}\left(\frac{1}{d}+\frac{Pb_i-z_i}{dq}+\frac{d-1-h_i}{d}\right)
\\ &= \sum_{0\leq h_i<d}B_{\tau_i+1}\left(\frac{1}{d}+\frac{Pb_i-z_i}{dq}+\frac{h_i}{d}\right).
\end{split}
\end{equation*}
We apply the multiplication theorem for Bernoulli numbers (\ref{Bernoulli}) and obtain 
\begin{equation*}
\begin{split}
\sum_{0\leq h_i<d}\bet_{\tau_i+1}\left(\frac{Pb_i-z_i-qh_i}{dq}\right)&= d^{-\tau_i} B_{\tau_i+1}\left(1+\frac{Pb_i-z_i}{q}\right) \\ &= d^{-\tau_i} \bet_{\tau_i+1}\left(\frac{Pb_i-z_i}{q}\right).
\end{split}
\end{equation*}
Similarly the same computation holds for $i\in I_2$ with $b_i$ replaced by $a_i$. Hence we obtain
\begin{equation*}
\begin{split}
S_{(I_1,I_2,\bftau)}(P;dr,dq)=& \varrho_{(I_1,I_2,\bftau)}d^{|I_3|-|\bftau|} \sum_{\bfz \mmod d} e\left(\frac{r}{q}F(\bfz)\right) \prod_{i\in I_1} \bet_{\tau_i+1}\left(\frac{Pb_i-z_i}{q}\right)\\ &\prod_{i\in I_2}\bet_{\tau_i+1}\left(\frac{Pa_i-z_i}{q}\right) \\ 
= &d^{|I_3|-|\bftau|}S_{(I_1,I_2,\bftau)}(P;r,q).
\end{split}
\end{equation*}
This completes the proof of the lemma.

\end{proof}

Next we consider a truncated piece of the singular series $\grS_{(I_1,I_2,\bftau)}(P,n)$ in the case where it is absolutely convergent. Under the assumptions of Lemma \ref{singseries} we have 
\begin{equation*}
\grS_{(I_1,I_2,\bftau)}(P,n)= \sum_{q|Q!}\sum_{\substack{r=1\\ (r,q)=1}}^q q^{-|I_3|+|\bftau|}S_{(I_1,I_2,\bftau)}(P;r,q)e\left(-\frac{r}{q}n\right)+O(Q^{-\del}).
\end{equation*}
For some $q$ appearing in the summation we let $d$ be defined by $qd=Q!$. Using Lemma \ref{lll} we rewrite the sum on the right hand side as
\begin{equation}\label{aux9}
\begin{split}
\grS_{(I_1,I_2,\bftau)}(P,n)=&\sum_{q|Q!} \sum_{\substack{r=1\\ (r,q)=1}}^q (qd)^{-|I_3|+|\bftau|}S_{(I_1,I_2,\bftau)}(P;rd,qd) e\left(\frac{-rd}{Q!}n\right) +O(Q^{-\del})\\ &= \sum_{r'=1}^{Q!} (Q!)^{-|I_3|+|\bftau|}S_{(I_1,I_2,\bftau)}(P;r',Q!)e\left(-\frac{r'}{Q!}n\right)+O(Q^{-\del}).
\end{split}
\end{equation}
Note that by orthogonality one has
\begin{equation*}
\sum_{r=1}^{Q!} e\left(\frac{r}{Q!}(F(\bfz)-n)\right)= \left\{ \begin{array}{cc} Q!& \mbox{ if } F(\bfz)-n\equiv 0 \mmod Q! \\ 0 & \mbox{ otherwise}.\end{array}\right.
\end{equation*}
Let
\begin{equation*}
\bet_\bftau(\bfz,Q!)= \varrho_{(I_1,I_2,\bftau)}\prod_{i\in I_1} \bet_{\tau_i+1} \left(\frac{Pb_i-z_i}{Q!}\right)\prod_{i\in I_2} \bet_{\tau_i+1} \left(\frac{Pa_i-z_i}{Q!}\right).
\end{equation*}
If we use the definition of $S_{(I_1,I_2,\bftau)}(P;r',Q!)$ in the last sum in (\ref{aux9}), we see that
\begin{equation*}
\grS_{(I_1,I_2,\bftau)}(P,n)= (Q!)^{-|I_3|+|\bftau|+1} \sum_{0\leq \bfz <Q!} \mathbf{1}_{\{F(\bfz)\equiv  n\mmod Q!\}} \bet_\bftau(\bfz,Q!)+O(Q^{-\del}).
\end{equation*}

We state our observations in the following lemma.

\begin{lemma}\label{S2}
Let $(I_1,I_2,\bftau)\in \calI (K)$, and assume that 
\begin{equation*}
2^{-d+1}\frac{s-\sig}{d-1}>K+1.
\end{equation*}
Let $Q\geq 1$. Then there is some $\del >0$ such that
\begin{equation*}
\grS_{(I_1,I_2,\bftau)}(P,n) = (Q!)^{-|I_3|+|\bftau|+1} \sum_{0\leq \bfz <Q!} \mathbf{1}_{\{F(\bfz)\equiv  n\mmod Q!\}} \bet_\bftau(\bfz,Q!) +O(Q^{-\del}).
\end{equation*}
\end{lemma}

One can interpret the expression for $\grS_{(I_1,I_2,\bftau)}(P,n)$ in Lemma \ref{S2} as a weighted version of the counting function $F(\bfz)\equiv n$ modulo $Q!$. We use Lemma \ref{S2} to further understand the singular series $\grS_{(I_1,I_2,\bftau)}(P,n)$, which occurr in the terms of largest order directly after the main term in the expansion in Theorem \ref{thm1}. They correspond to situations where $I_1\cup I_2=\{i_0\}$ for a single element $1\leq i_0\leq s$ and $|\bftau|=0$. We assume in the following that $Q!|Pb_{i_0}$ or $Q!|Pa_{i_0}$ depending on whether $i_0\in I_1$ or $i_0\in I_2$. Under some symmetry assumptions on the form $F(\bfx)$ we can rewrite $\grS_{(I_1,I_2,\bftau)}(P,n)$ as a local density, up to a small error, and in particular determine its sign.

\begin{lemma}\label{S3}
In addition to the assumptions in Lemma \ref{S2}, let $I_1\cup I_2=\{i_0\}$, $|\bftau|=0$ and $Q!|Pa_{i_0}$ if $i_0\in I_2$ or $Q!|Pb_{i_0}$ if $i_0\in I_1$. Furthermore assume that the counting function
\begin{equation*}
\begin{split}
r(z_{i_0},Q!,n):=\sharp\{ \bfz_{I_3} \mmod Q!: F(\bfz_{I_3},z_{i_0})\equiv n \mmod Q!\} 
\end{split}
\end{equation*}
satisfies $r(z_{i_0},Q!,n)=r(-z_{i_0},Q!,n)$ for all $z_{i_0}$ modulo $Q!$. Then one has
\begin{equation*}
\grS_{(I_1,I_2,\bftau)}(P,n)=\frac{1}{2}(-1)^{|I_1|+1}(Q!)^{-s+1} r(0,Q!,n)+O(Q^{-\del}),
\end{equation*}
for some $\del >0$.
\end{lemma}

\begin{proof}
By Lemma \ref{S2} we can express a truncated version of the singular series $\grS_{(I_1,I_2,\bftau)}(P,n)$ as 
\begin{equation*}
\grS_{(I_1,I_2,\bftau)}(P,n)= (-1)^{|I_1|}(Q!)^{-s+1} \sum_{0\leq z_{i_0}<Q!}\bet_1\left(\frac{-z_{i_0}}{Q!}\right) r(z_{i_0},Q!,n) +O(Q^{-\del}).
\end{equation*}
Recall that the first Bernoulli polynomial $B_1(x)$ is defined as $B_1(x)=x-\frac{1}{2}$. We hence rewrite the last expression as
\begin{equation}\label{aux3}
\begin{split}
\grS_{(I_1,I_2,\bftau)}(P,n)= &(-1)^{|I_1|}(Q!)^{-s+1}\bet_1(0)r(0,Q!,n) \\ &+ \sum_{0< z_{i_0}<Q!}\left(\frac{1}{2}-\frac{z_{i_0}}{Q!}\right) r(z_{i_0},Q!,n) +O(Q^{-\del}).
\end{split}
\end{equation}
By assumption we have $r(z_{i_0},Q!,n)=r(Q!-z_{i_0},Q!,n)$ and we observe that
\begin{equation*}
\frac{1}{2}-\frac{z_{i_0}}{Q!}+\frac{1}{2}-\frac{Q!-z_{i_0}}{Q!}=0.
\end{equation*}
Hence the second sum in (\ref{aux3}) vanishes and we obtain
\begin{equation*}
\grS_{(I_1,I_2,\bftau)}(P,n)= (-1)^{|I_1|}(Q!)^{-s+1}\bet_1(0)r(0,Q!,n)+O(Q^{-\del}),
\end{equation*}
as desired.
\end{proof}

We remark that Lemma \ref{S2} is useful in determining the sign of $\grS_{(I_1,I_2,\bftau)}(P,n)$ and showing that these singular series are non-zero under certain conditions. The counting function $r(z_{i_0},Q!,n)$ is always non-negative, and furthermore, the term $(Q!)^{-s+1} r(0,Q!,n)$ can be shown to be positive under the assumption of the existence of non-singular local solutions for all finite primes. As an example, we compare this to Theorem 1.4 in \cite{VauWooA14} for the special case of $F(\bfx)=\sum_{i=1}^s x_i^d$. As we shall see in the next section one has
\begin{equation*}
\grS_{\emptyset,I_2,\mathbf{0}}(P,n)= \grS_{s,|I_2|}(n),
\end{equation*}
where the singular series on the right hand side is defined as in Theorem 1.4 in \cite{VauWooA14}. In the case where $|I_2|=1$ and $Q!|n$, the symmetry assumption on $r(z_{i_0},Q!,n)$ in Lemma \ref{S3} is satisfied and this shows that
\begin{equation*}
\grS_{s,|I_2|}(n)= -\frac{1}{2}(Q!)^{-s+1} r(0,Q!,n)+O(Q^{-\del}).
\end{equation*}
Observe that $(Q!)^{-s+1} r(0,Q!,n)$ converges to $\grS_{s-1}(n)$ for $Q\rightarrow \infty$ (by applying the Chinese remainder theorem) and hence we recover, up to a less precise error term, the result of Theorem 1.4 in \cite{VauWooA14}. We note that our assumptions on $s$ are of course much stronger than those in \cite{VauWooA14}, since we applied a result for a very general form $F(\bfx)$ to the sum of $s$ $d$-th powers. However, we could feed our method on the minor arcs with mean value estimates for sums of $d$th powers instead, and recover results of comparable strength in $s$.

\section{Comparison to work of Vaughan and Wooley}\label{comparison}
In this section we compare our main theorem \ref{thm1} with the results of Vaughan and Wooley \cite{VauWooA14} in the case $F(\bfx)=\sum_{i=1}^s x_i^d$. We note that this is the only polynomial that Vaughan and Wooley consider in their paper and their methods are very much adapted to using the diagonal structure of the equations, not only on the minor, but also on the major arcs. For $d>1$ and a positive natural number $n$ they study the counting function $R_s^{VW}(n)$, which is defined to be the number of representations of $n$ as the sum of $s$ $d$th powers of positive integers. In our notation one has $R_s^{VW}(n)=R_\calB(P,n)$ with $\calB = (0,1]^s$ and $P=2n^{1/d}$.\par
The study of the counting function $R_s^{VW}(n)$ has a long history which we do not aim to repeat here. However, only very recently Vaughan and Wooley \cite{VauWooA14} for the first time gave an asymptotic expansion for $R_s^{VW}(n)$ which has the form of a higher order expansion, i.e. describes not only the first leading term but also lower order terms. Let $J$ be some non-negative integer. If $d$ is even and $s$ sufficiently large depending on $J$ and $d$, then Theorem 1.1 in \cite{VauWooA14} states that
\begin{equation}\label{VW3}
R_s^{VW}(n)= n^{s/d-1}\left(\grC_0+\grC_1n^{-1/d}+\ldots + \grC_J n^{-J/d}\right)+o(n^{(s-J)/d-1}),
\end{equation}
where the constants $\grC_j$ are specified in equation (1.6) in \cite{VauWooA14} for $0\leq j\leq J$. In detail, one has
\begin{equation*}
\grC_j= \left(-\frac{1}{2}\right)^j \binom{s}{j} \frac{\Gam(1+1/d)^{s-j}}{\Gam((s-j)/d)} \grS_{s-j}(n),
\end{equation*}
where the singular series $\grS_{s-j}(n)$ is defined as
\begin{equation*}
\grS_s(n)= \sum_{q=1}^\infty \sum_{\substack{a=1\\ (a,q)=1}}^q \left(q^{-1}S(q,a)\right)^s e(-na/q),
\end{equation*}
with 
\begin{equation*}
S(q,a)=\sum_{r=1}^q e(ar^d/q).
\end{equation*}
In the case of odd $d$ Vaughan and Wooley prove the asymptotic expansion (\ref{VW3}) for $0\leq J \leq d$ and $s$ sufficiently large depending on $d$ and $J$, with constants of the form
\begin{equation*}
\grC_j= \binom{s}{j} \frac{\Gam(1+1/d)^{s-j}}{\Gam((s-j)/d)}\grS_{s,j}(n).
\end{equation*}
Here the singular series $\grS_{s,j}(n)$ is defined as
\begin{equation*}
\grS_{s,j}(n)= \sum_{q=1}^\infty \sum_{\substack{a=1\\ (a,q)=1}}^q (q^{-1}S(q,a))^{s-j} T(q,a)^j e(-na/q),
\end{equation*}
with
\begin{equation*}
T(q,a)= \sum_{r=1}^q \left(\frac{1}{2}-\frac{r}{q}\right)e(ar^d/q).
\end{equation*}
We note that in the case of $d$ even one has
\begin{equation*}
T(q,a)=-\frac{1}{2}+\sum_{r=1}^{q-1}\left(\frac{1}{2}-\frac{r}{q}\right)e(ar^d/q) = -\frac{1}{2}.
\end{equation*}
Hence the definition of $\grC_j$ for the case of $d$ odd also covers the case of $d$ even.\par
Our Theorem \ref{thm1} applied to $R_\calB (P,n)$ and $K=J+1$ delivers for $s$ sufficiently large the asymptotic expansion
\begin{equation}\label{VW6}
\begin{split}
R(P,n)&= \sum_{(I_1,I_2,\bftau)\in \calI(J+1)} \grS_{(I_1,I_2,\bftau)}(P,n)\calJ_{(I_1,I_2,\bftau)} (P^{-d}n) P^{|I_3|-|\bftau|-d} \\ &+O\left( P^{s-d-J-\del}\right),
\end{split}
\end{equation}
for some $\del >0$. We note that our assumptions on $s$ in Theorem \ref{thm1} are much stronger than those in Theorem 1.1 and Theorem 1.2 in \cite{VauWooA14}. This should not be surprising, since our main theorem is mainly aimed at general forms $F(\bfx)$ and hence we use on the minor arcs the bounds for exponential sum from the work of Birch \cite{Bir62}. We could implement the minor arc bounds used in \cite{VauWooA14} in our work for the special situation of $F(\bfx)=\sum_{i=1}^s x_i^d$, and obtain the same restrictions on $s$ with our methods.\par
In the following we aim to show that the asymptotic expansion (\ref{VW6}) coincides with the one given by Vaughan and Wooley of the form (\ref{VW3}). We note that our asymptotic expansion looks more technical than the one of Vaughan and Wooley, since it covers general forms $F(\bfx)$ and does not need the restriction $J\leq d$ as in Theorem 1.2 in \cite{VauWooA14}. We will shortly see that most of the terms in our expansion (\ref{VW6}) vanish in the special case of Waring's problem. However, these additional terms are necessary to cover the general case of representing an integer $n$ by a general form $F(\bfx)$ of degree $d$.\par
We recall that
$$
\grS_{(I_1,I_2,\bftau)}(P,n)=\sum_{q=1}^\infty \sum_{\substack{r=1 \\ (r,q)=1}}^q q^{-|I_3|+|\bftau|}e\left(\frac{-r}{q}n\right)S_{(I_1,I_2,\bftau)}(P;r,q),
$$
with
\begin{equation*}
\begin{split}
S_{(I_1,I_2,\bftau)}(P;r,q)=&\sum_{0\leq \bfz<q}e\left(\frac{r}{q}F(\bfz)\right) \left(\prod_{i\in I_1} \frac{(-1)^{\tau_i+1}}{(\tau_i+1)!}\bet_{\tau_i+1}\left(\frac{Pb_i-z_i}{q}\right)\right) \\ &\left(\prod_{i\in I_2} \frac{(-1)^{\tau_i}}{(\tau_i+1)!}\bet_{\tau_i+1}\left(\frac{Pa_i-z_i}{q}\right)\right) .
\end{split}
\end{equation*}
We first consider the case where $I_1=\emptyset$ and $\tau_i=0$ for all $1\leq i\leq s$. Then the description of this exponential sum reduces to
\begin{equation*}
S_{(\emptyset,I_2,\mathbf{0})}(P;r,q)= \left(\sum_{z =1}^q \bet_1\left(\frac{-z}{q}\right)e\left(\frac{r}{q}z^d\right)\right)^{|I_2|} S(q,r)^{s-|I_2|}.
\end{equation*}
We recall that %$B_1(x)=x-\frac{1}{2}$ and 
$\bet_1(x)=\left\{x\right\}-\frac{1}{2}$, and hence we deduce that
\begin{equation*}
\begin{split}
\sum_{z=1}^q \bet_1\left(\frac{-z}{q}\right)e\left(\frac{r}{q}z^d\right)
%& = \sum_{z=1}^q B_1\left(1-\frac{z}{q}\right)e\left(\frac{r}{q}z^d\right) \\ &
= \sum_{z=1}^q \left(\frac{1}{2}-\frac{z}{q}\right)e\left(\frac{r}{q}z^d\right) = T(q,r).
\end{split}
\end{equation*}
Hence we obtain
\begin{equation*}
\grS_{(\emptyset,I_2,\mathbf{0})}(P,n)= \sum_{q=1}^\infty \sum_{\substack{r=1\\ (r,q)=1}}^q \left(q^{-1}S(q,r)\right)^{s-|I_2|}T(q,r)^{|I_2|}e\left(-\frac{r}{q}n\right) = \grS_{s,|I_2|}(n).
\end{equation*}
Note that $\calJ_{(\emptyset,I_2,\mathbf{0})}(n)$ is the same as the traditional singular integral with respect to the equation $\sum_{i=1}^{s-|I_2|}x_i^d=n$ and the box $(0,1]^{s-|I_2|}$. It is well known that
\begin{equation*}
\calJ_{(\emptyset,I_2,\mathbf{0})}(1)=\frac{\Gam(1+1/d)^{s-|I_2|}}{\Gam((s-|I_2|)/d)},
\end{equation*}
see for example the proof of Theorem 4.1 in \cite{Dav2005}. Furthermore, we note that $\calJ_{(\emptyset,I_2,\mathbf{0})}(t)$ has a scaling property in $t>0$ of the form
\begin{equation*}
\begin{split}
\calJ_{(\emptyset,I_2,\mathbf{0})}(t)&=\int_\R e(-\gam t) \left(\int_0^1 e(\gam x^d)\d x\right)^{s-|I_2|}\d \gam \\ &= t^{-1}\int_\R e(-\gam ) \left(\int_0^1 e(t^{-1}\gam x^d)\d x\right)^{s-|I_2|}\d \gam\\ &=t^{(s-|I_2|)/d-1}\int_\R e(-\gam ) \left(\int_0^{t^{-1/d}} e(\gam x^d)\d x\right)^{s-|I_2|}\d \gam.
\end{split}
\end{equation*}
Interpreting this expression as in the proof of Theorem 4.1 in \cite{Dav2005} as a volume integral shows that
\begin{equation*}
\begin{split}
\calJ_{(\emptyset,I_2,\mathbf{0})}(P^{-d}n)&= \calJ_{(\emptyset,I_2,\mathbf{0})}(2^{-d}) =2^{d-(s-|I_2|)} \int_\R e(-\gam ) \left(\int_0^{2} e(\gam x^d)\d x\right)^{s-|I_2|}\d \gam\\ &=  2^{d-(s-|I_2|)}\frac{\Gam(1+1/d)^{s-|I_2|}}{\Gam((s-|I_2|)/d)}.
\end{split}
\end{equation*}
If we consider some fixed $0\leq j\leq J$, then there are exactly $\binom{s}{j}$ terms in the sum in (\ref{VW6}) with $|I_2|=j$ and $I_1=\emptyset$ and $\bftau=0$. All of them are the same and equal to
\begin{equation*}
\grS_{(\emptyset,I_2,\mathbf{0})}(P,n)\calJ_{(\emptyset,I_2,\mathbf{0})}(P^{-d}n)P^{s-|I_2|-d} = \grS_{s,j}(n) \frac{\Gam(1+1/d)^{s-j}}{\Gam((s-j)/d)}n^{(s-j)/d-1}.
\end{equation*}
Hence they exactly match the asymptotic expansion of Vaughan and Wooley as in (\ref{VW3}). In order to show that the two asymptotic expansions in (\ref{VW3}) and (\ref{VW6}) coincide, it hence remains to show that all the terms in \ref{VW6} with $I_1\neq \emptyset$ or $\bftau\neq \mathbf{0}$ vanish.\par
We start with the case where $I_1\neq \emptyset$. If we assume that (\ref{integralconv1}) holds, then equation \ref{SingIntVol} shows that $\calJ_{(I_1,I_2)}(P^{-d}n;\bfx_{I_1},\bfx_{I_2})=0$ for $\bfx_{I_1}$ in a sufficiently small neighbourhood of the constant vector $\mathbf{1}$ and $\bfx_{I_2}$ in a small neighbourhood of the constant vector $\mathbf{0}$. Indeed, the volume of the hypersurface $F(\bfx)=2^{-d}$ in the cube $\prod_{i\in I_3}[0,1]$ is zero for $\bfx_{I_1}$ close to $\mathbf{1}$ and $\bfx_{I_2}$ close to $\mathbf{0}$. Hence Lemma \ref{SingIntb} shows that 
\begin{equation*}
\calJ_{(I_1,I_2,\bftau)}(P^{-d}n)=0,
\end{equation*}
as soon as $I_1\neq \emptyset$.\par
Next we consider the case where $I_1=\emptyset$ and $\bftau \neq \mathbf{0}$. Note that then we may assume that there is some $i_0\in I_2$ with $\tau_{i_0}>0$. Furthermore, we observe that
\begin{equation*}
\frac{\partial^\nu}{\partial x^\nu}e(\gam x^d)|_{x=0}=0,\quad \mbox{ for } \quad 1\leq \nu <d.
\end{equation*}
For odd degree $d$, Vaughan and Wooley assume that $J\leq d$, which implies that $|\tau_{i_0}|<d$. Hence we obtain in this situation $f^{(\bftau)}(\gam,\sig_{\bfa,\bfb}(\bfx))=0$, as a function in $\bfx_{I_3}$ and hence $J_{(I_1,I_2,\bftau)}(\gam)=0$ and $\calJ_{(I_1,I_2,\bftau)}(P^{-d}n)=0$.\par
Finally, we need to consider the case where $d$ is even and $I_1=\emptyset$ and $\bftau\neq \mathbf{0}$. Hence we may assume that $\tau_i>0$ for some $i\in I_2$. First assume that $\tau_i$ is even. Then we have $\bet_{\tau_i +1}(0)=0$ and the symmetry relation $\bet_{\tau_i +1}(x)= - \bet_{\tau_i +1}(1-x)$ holds. Hence we see that
\begin{equation*}
\sum_{0\leq z<q} e\left(\frac{r}{q}z^d\right)\bet_{\tau_i +1}\left(\frac{-z}{q}\right) = \sum_{0\leq z\leq q} e\left(\frac{r}{q}z^d\right)\bet_{\tau_i +1}\left(\frac{-z}{q}\right) =0.
%= \frac{1}{2}(\sum_{0\leq z<q} e\left(\frac{r}{q}z^d\right)\bet_{\tau_i +1}\left(\frac{-z}{q}\right) +\sum_{0\leq z<q} e\left(\frac{r}{q}(q-z)^d\right)\bet_{\tau_i +1}\left(\frac{-(q-z)}{q}\right) 
\end{equation*}
This implies that $\grS_{(I_1,I_2,\bftau)}(P,n)=0$. In the case where $\tau_i$ is odd, we observe that 
\begin{equation*}
\frac{\partial^{\tau_i}}{\partial x^{\tau_i}}e(\gam x^d)|_{x=0}=0,
\end{equation*}
since we have assumed $d$ to be even. This shows that $\calJ_{(I_1,I_2,\bftau)}(P^{-d}n)=0$.\par
%tow ways to see that: either interpret it as the derivative of an even function at x=0. Or understand in the expressions \sum_i P_i(x)e(\gam x^d) that the P_i(x) are all even or odd polys depnding on the parity of the derivative tau_i.
We hence conclude that the asymptotic expansions in (\ref{VW3}) and (\ref{VW6}) are indeed the same, i.e. our Theorem \ref{thm1} reduces in the case $F(\bfx)=\sum_{i=1}^s x_i^d$ to the multi-term expansion computed by Vaughan and Wooley in \cite{VauWooA14}.

\bibliographystyle{amsbracket}

\begin{thebibliography}{18}


\bibitem{Bir62}
B. J. Birch, \emph{Forms in many variables}, Proc. Roy. Soc. Ser. A \textbf{265} (1962), 245--263.

\bibitem{BroHB09}
T. D. Browning and D. R. Heath-Brown, \emph{Rational points on quartic hypersurfaces}, J. reine angew. Math. \textbf{629} (2009), 37--88.

\bibitem{BroPreA14}
T. D. Browning and S. M. Prendiville, \emph{Improvements in Birch's theorem on forms in many variables}, J. reine angew. Math, to appear. 

%\bibitem{BraA13}
%J. Brandes, \emph{Forms representing forms and linear spaces on hypersurfaces}, Proc. London Math. Soc, to appear. (arXiv:12202.5026)

%\bibitem{DieA12}
%R. Dietmann, \emph{Weyl's inequality and systems of forms}, preprint. (arXiv:1208.1968)

%\bibitem{Dav1959}
%H. Davenport, \emph{Cubic Forms in Thirty-Two Variables},
%Phil. Trans. R. Soc. Lond. A \textbf{251} (1959), 193--232.

\bibitem{Dav2005}
H. Davenport, \emph{Analytic methods for Diophantine equations and Diophantine inequalities}, Cambridge Mathematical Library. Cambridge University Press,
Cambridge, second edition, 2005. With a foreword by R. C. Vaughan,
D. R. Heath-Brown and D. E. Freeman, Edited and prepared for publication by
T. D. Browning.

\bibitem{HB83}
D. R. Heath-Brown, \emph{Cubic forms in ten variables}, Proc. London Math. Soc \textbf{47} (1983), 225--257.

\bibitem{HB96}
D. R. Heath-Brown, \emph{A new form of the circle method, and its application to quadratic forms}, J. reine angew. Math. \textbf{481} (1996), 149--206. 

\bibitem{HB07}
D. R. Heath-Brown, \emph{Cubic forms in 14 variable}, Invent. Math. {\bf 170} (2007), 199--230.

\bibitem{How95}
F. T. Howard, \emph{Applications of a Recurrence for the Bernoulli Numbers}, J. Number Theory {\bf 52} (1996), 157--172.


\bibitem{Loh96}
W. K. A. Loh, \emph{Limitation to the asymptotic formula in Waring's problem}, Acta Arith. {\bf 74} (1996), no. 1, 1--15.

%\bibitem{Schmidt1981}
%W. M. Schmidt, \emph{Simultaneous rational zeros of quadratic forms}, Seminar Delange-%
%Pisot-Poitou 1981. Progress in Math. Vol 22 (1982), 281--307.

%\bibitem{LeeAlan}
%S.-L. A. Lee, \emph{Birch's theorem in function fields}, submitted, 2012. (arXiv: 1109.4953)

%\bibitem{bihomforms}
%D. Schindler, \emph{Bihomogeneous forms in many variables}, J. Th\'eorie Nombres Bordeaux, to appear. (arXiv:1301.6516)

\bibitem{Schmidt80}
W. M. Schmidt, \emph{Simultaneous rational zeros of quadratic forms}, Seminar of Number Theory, Paris 1980-81, Progr. Math \textbf{22} (1982), 281--307.

%\bibitem{Schmidt82}
%W. M. Schmidt, \emph{On cubic polynomials. IV. Systems of rational equations}, Monatsh. Math. \textbf{93} (1982), 329--348.

\bibitem{Schmidt85}
W. M. Schmidt, \emph{The density of integer points on homogeneous varieties},
Acta Math. \textbf{154} (1985), no. 3-4, 243--296.

%\bibitem{Ski97}
%C.M. Skinner, \emph{Forms over number fields and weak approximation}, 
%Comp. Math. \textbf{106} (1997), 11--29.

\bibitem{SWD}
Sir P. Swinnerton-Dyer, \emph{Counting points on cubic surfaces. II. Geometric methods in algebra and number theory}, 303-–309, 
Progr. Math., 235, Birkh{\"a}user Boston, Boston, MA, 2005. 


%\bibitem{Vau1997}
%R. C. Vaughan, \emph{The Hardy-Littlewood method}, volume 125 of Cambridge
%Tracts in Mathematics, Cambridge University Press, Cambridge, second edition, 1997.

\bibitem{VauWooA14}
R. C. Vaughan and T. D. Wooley, \emph{The asymptotic formula in Waring's problem: higher order expansions}, preprint, arXiv:1309.0443.

\end{thebibliography}
\providecommand{\bysame}{\leavevmode\hbox to3em{\hrulefill}\thinspace}

\end{document}